\documentclass[11pt]{amsart}

\usepackage{amsmath,amssymb,latexsym,soul,cite,mathrsfs}

\usepackage{color,enumitem,graphicx}
\usepackage[colorlinks=true,urlcolor=blue,
citecolor=red,linkcolor=blue,linktocpage,pdfpagelabels,
bookmarksnumbered,bookmarksopen]{hyperref}
\usepackage[english]{babel}

\usepackage[left=2.9cm,right=2.9cm,top=2.8cm,bottom=2.8cm]{geometry}
\usepackage[hyperpageref]{backref}

\usepackage[colorinlistoftodos]{todonotes}
\usepackage{dsfont}
\makeatletter
\providecommand\@dotsep{5}
\def\listtodoname{List of Todos}
\def\listoftodos{\@starttoc{tdo}\listtodoname}
\makeatother

\numberwithin{equation}{section}

\newtheorem{theorem}{Theorem}[section]
\newtheorem{proposition}[theorem]{Proposition}
\newtheorem{lemma}[theorem]{Lemma}
\newtheorem{corollary}[theorem]{Corollary}
\newtheorem{example}[theorem]{Example}
\newtheorem{claim}[theorem]{Claim}

\newtheorem{remark}{Remark}
\newtheorem{definition}[theorem]{Definition}

\newcommand\R{\mathbb R}
\newcommand\N{\mathbb N}

\begin{document}
	
	\title[New minimax theorems for...]
	{New minimax theorems for lower semicontinuous functions and applications}

	\author{Claudianor O. Alves}
	\author{Giovanni Molica Bisci}
	\author{Ismael S. da Silva$^*$}

	\address[Claudianor O. Alves]{\newline\indent Unidade Acad\^emica de Matem\'atica
		\newline\indent
		Universidade Federal de Campina Grande,
		\newline\indent
		58429-970, Campina Grande - PB - Brazil}
	\email{\href{mailto:coalves@mat.ufcg.edu.br}{coalves@mat.ufcg.edu.br}}

\address[Giovanni Molica Bisci]{Dipartimento di Scienze Pure e Applicate (DiSPeA), Universit\`a degli Studi
	\newline\indent
	di Urbino 	Carlo Bo, Piazza della Repubblica 13, 61029 Urbino (Pesaro e Urbino), Italy}
\email{\href{giovanni.molicabisci@uniurb.it}{giovanni.molicabisci@uniurb.it}}

	\address[Ismael S. da Silva]
	{\newline\indent Unidade Acad\^emica de Matem\'atica
		\newline\indent
		Universidade Federal de Campina Grande,
		\newline\indent
		58429-970, Campina Grande - PB - Brazil}
	\email{\href{ismael.music3@gmail.com}{ismael.music3@gmail.com}}

	\pretolerance10000
	
	
	\begin{abstract}
	\noindent In this paper we prove a version of the Fountain Theorem for a class  of nonsmooth functionals that are sum of a $C^1$ functional and a convex lower semicontinuous functional, and also a version of a theorem due to Heinz for this class of functionals. The new abstract theorems will be used to prove the existence of infinitely many solutions for some elliptic problems whose the associated energy functional is of the above mentioned type. We study problems with logarithmic nonlinearity and a problem involving the 1-Laplacian operator.
\end{abstract}

\thanks{Claudianor Alves was partially supported by  CNPq/Brazil 307045/2021-8 and Projeto Universal FAPESQ 3031/2021, e-mail:coalves@mat.ufcg.edu.br}
\thanks{Ismael da Silva is the corresponding author and he was partially supported by  CAPES, Brazil.}

\maketitle

\section{Introduction}
Several studies in critical point theory concern the existence and multiplicity of solutions for elliptic problems, in this line we have, for example, the famous Bartsch's Fountain Theorem (see \cite[Theorem 3.6]{Willem} and also \cite{Bartsch0}), which ensures the existence and multiplicity of critical points for functionals of class $C^1$ that satisfy some condition of symmetry.

Recently, many works have been focused on to establish variant Fountain Theorems: in 1995, Bartsch-Willem \cite{Bartsch1} proved a dual version for the Fountain Theorem, which was used to show the multiplicity of solutions for a class of problems with concave and convex nonlinearities. Posteriorly, in 2013, Batkam-Colin \cite{Batkam} established a version of Fountain Theorem for indefinite functionals $\varphi \in C^1(X,\mathbb{R})$ ($X$ is a Banach space) by using $\tau$-topology and the degree theory of Kryszewski-Szulkin (see \cite[Chapter 6]{Willem}). This result was improved in 2016 by Gu-Zhou \cite{Gu}. In the approach used in \cite{Batkam}, it is assumed the $\tau$-upper semi-continuity of $\varphi$, while in \cite{Gu}, this condition is not assumed for $\varphi$. We would also like to cite Zou \cite{Zou}, where it was established a version of Fountain Theorem without the classical Palais-Smale's condition. In that paper, it was showed the existence of infinitely many solutions for elliptic problems under no Ambrosetti-Rabinowitz's superquadraticity condition on the nonlinearities. The Fountain Theorem and its variants have been applied in a lot of works, for the reader interested in this subject, we would like to cite \cite{Bartsch1, Batkam, Gu, Liu, Sun, Zou}.

In 1981, Chang developed a study that permits to apply variational method to functionals $\varphi \in Lip_{loc}(X,\mathbb{R})$, i.e, functionals that are locally Lipschitz (see \cite{Chang1}). Using the notion of subdifferential of a locally Lipschitz functional (see Section 2 below), Chang introduced some minimax principles such as a version of the Mountain Pass type for functionals of class $Lip_{loc}(X, \mathbb{R})$. Later, in 2010, Dai \cite{Dai} showed a version of the Bartsch's Fountain Theorem  for functionals $\varphi \in Lip_{loc}(X,\mathbb{R})$ by following the notions found in \cite{Chang1}.

In 1986, Szulkin \cite{Szulkin} enlarged the variational methods for another class  of functionals. More specifically, Szulkin considered functionals belonging to the following class
\begin{itemize}
\item[$(H_0)$] \textit{$I:X \rightarrow (-\infty, \infty],$ $I:=\Phi+\Psi,$ with $\Phi \in C^1(X,\mathbb{R})$ and $\Psi:X\rightarrow (-\infty,\infty]$ is a convex lower semicontinuous functional and $\Psi \not\equiv \infty$.}
\end{itemize}

 A succinct description of the notion of critical point introduced by Szulkin is done in the Section 2 of this work. In \cite{Szulkin}, Szulkin provided a list of minimax results for functionals $I \in (H_0)$ such as, Mountain Pass type theorem, versions of results due to Clark, and Ambrosetti-Rabinowitz type results related to the multiplicity of critical points of even functionals (see also \cite{Motreanu1, Motreanu}). In the references \cite{Alves-de Morais, Alves-Ji,Ji-Szulkin, Kobayashi, Mancini, Squassina-Szulkin} the reader can see how the theory developed in \cite{Szulkin} provides a powerful tool to solve elliptic problems whose associated energy functional is not of class $C^1$.

Motivated by the facts above, the main purpose of this work is to give an answer for the following question:
\\ \\
\textbf{Question 1}: Is it possible to prove a Fountain-type Theorem for functionals $I\in(H_0)$?
\\

In Ji-Szulkin \cite{Ji-Szulkin},  we can find a partial answer for the question above. In that work, the authors studied the existence of multiple solutions for the following logarithmic Schr\"odinger equation
\begin{equation}\label{Eq1}
-\Delta u + V(x)u= u \log u^2,\,\,\,\, x \in \mathbb{R}^N,
\end{equation}
where $V$ is a continuous function and $\displaystyle \lim_{|x| \to +\infty}V(x)=+\infty$, for more details see \cite[Section 1]{Ji-Szulkin}.

The energy functional associated with problem (\ref{Eq1}) is given by

\begin{equation}\label{FunJ}
J_0(u) := \frac{1}{2}\int_{\mathbb{R}^N}(|\nabla u|^2+(V(x)+1)u^2)dx-\frac{1}{2}\int_{\mathbb{R}^N}u^2\log u^2 dx,\,\,\,\, u \in E,
\end{equation}
where $E$ is a convenient subspace of $H^1(\mathbb{R}^N)$. The functional $J_0$ can assume the value $+\infty$ for some $u \in H^1(\mathbb{R}^N)$, in particular, $J_0$ is not of class $C^1$. As made in \cite{Alves-de Morais, Alves-Ji, Squassina-Szulkin}, the functional $J_0$ can be decomposed as a sum of a functional $J_1$ of class $C^1$ and a convex lower semicontinuous functional $J_2:H^1(\mathbb{R}^N)\rightarrow (-\infty,\infty]$, so $J_0 \in (H_0)$. The framework used by the authors of \cite{Ji-Szulkin} consists in showing that $J_0$ satisfies the geometrical conditions of the Fountain Theorem and, by using the ideas presented in \cite{Squassina-Szulkin}, they established a particular deformation lemma that was used to show that the minimax levels associated with Fountain Theorem are critical values of $J_0$. Finally, it is very important to mention that the approach above works well, because the functional $J_0$ verifies the following key property: when the integrals in (\ref{FunJ}) are taken on a bounded domain of $\mathbb{R}^N$, then $J_0$ is a functional of class $C^1$; see Lemma \ref{Psi1} - Part $ii)$. This property is crucial to establish a special deformation lemma, see \cite[Proposition 3.4]{Ji-Szulkin} for more details.

In the present paper we improve the answer to Question 1, more precisely, we show that the answer to Question 1 is affirmative by proving a result that generalizes the Bartsch's Fountain Theorem. Our result is a version of the Bartsch's Fountain Theorem for functionals $I\in(H_0)$; see Theorem \ref{Fountain1}.

In order to prove our version of the Fountain Theorem, we will establish some results that complement the study made in \cite{Szulkin}. First of all, we would like to point out that we will present an equivariant version of the deformation lemma (see the Lemma \ref{EDL} in the sequel), which will be used to establish our main abstract result; see Theorem \ref{Fountain1}. The Bartsch's Fountain Theorem deals with a condition of symmetry that verifies a specific notion of admissibility; see $(G_0)$ in Section 2. Thus, in order to get a general Fountain Theorem, it was necessary to establish an  equivariant version of the deformation lemma involving a notion of admissibility. Our version complements the result of Szulkin in \cite[Corollary 2.4]{Szulkin} that only considered the particular case of antipodal action of $\mathbb{Z}_2$ on $X$.

A second question that has motivated the present paper is the following:\\ \\
\textbf{Question 2}: Is there a version of the dual Fountain Theorem for functionals $I\in(H_0)$?\\

Reading carefully the proof of  \cite[Theorem 3.18]{Willem}, we can note that the main idea of the Bartsch-Willem's dual version of Fountain Theorem consists in applying the usual version of Fountain Theorem to the functional $-I$, which permits to obtain a sequence $(c_j)$ of negative critical values of $I$ with $c_j \to 0.$ However, when  $I \in (H_0)$ we cannot repeat this procedure, because we do not know, in general, if the functional $-I$ belongs to $(H_0)$. Indeed, the condition $-I \in (H_0)$ only holds in the particular case that $\Psi$ and $-\Psi$ are convex and lower semicontinuous. In order to overcome this difficulty,  we present a version of a Heinz's Theorem, \cite[Proposition 2.2]{Heinz}, for functionals $I\in(H_0)$. The original result by Heinz provides an additional condition to a result by Clark (see \cite[Theorem 2.1]{Heinz}), which enables to assure the existence of a sequence $(c_j)$ of negative critical values converging to $0$, and so, in some sense, the Heinz's Theorem works like a dual Fountain Theorem, because it gives the same type of information. Using the notations introduced by Szulkin in \cite[Section 4]{Szulkin} we proved a version of Heinz's Theorem for functionals of the class $(H_0)$; see Theorem \ref{HT}.

Finally, by adapting the reasoning used in the proof of our version of Heinz's Theorem, we were able to prove one more abstract theorem involving the class $(H_0)$, see Theorem \ref{SP}, where it was proved that in \cite[Corollary 4.4]{Szulkin} the sequence of critical levels $c_j$ can be obtained such that $c_j \rightarrow \infty$.

The new minimax theorems mentioned above will be used to establish the existence of infinitely many solutions for some classes of elliptic problems involving a logarithmic nonlinearity. The first application considered in this paper is the following inclusion problem:
$$
\left\{\begin{aligned}
-&\Delta u + u + \partial F(x,u) \ni u\log u^2,  \;\;\mbox{a.e. in}\;\;\mathbb{R}^{N}, \\
& u \in H^{1}(\mathbb{R}^{N}),
\end{aligned}
\right. \leqno{(P_1)}
$$
where $F(x,t):= \displaystyle{\int_{0}^{t}}f(x,t)\,ds$ is a locally Lipschitz function, $f(\cdot,\cdot)$ is a $N$-measurable function satisfying some technical conditions and $\partial F(x,t)$ designates the generalized gradient of $F$ with respect to variable $t$ (see the Section 2 in the sequel for more details). The natural candidate for the energy functional associated with problem $(P_1)$ is given by
\begin{equation}\label{Functional}
I(u):=\frac{1}{2}\int_{\mathbb{R}^N}(|\nabla u|^2+|u|^2)dx + \int _{\mathbb{R}^N}F(x,u)dx - \int_{\mathbb{R}^N}L(u)dx,\,\,\,  u \in H^1(\mathbb{R}^N),
\end{equation}
where
$$
L(t):=\displaystyle{\int_{0}^{t}}s\log s^2\,ds=-\frac{t^2}{2}+\frac{t^2\log t^2}{2}, \quad t \in \mathbb{R}.
$$
However, the reader is invited to see that this functional does not belong to $C^1(H^1(\mathbb{R}^N),\mathbb{R})$, actually, the logarithmic term in (\ref{Functional}) gives the possibility that $I(u)=\infty$, for some $u \in H^1(\mathbb{R}^N)$. Having this in mind, we will prove that $I$ can be decomposed as sum of a $C^1$ functional $\Phi$ and a convex lower semicontinuous functional $\Psi$, which will permit to investigate the existence of multiple critical points of $I$ in the sense of the theory of Szulkin \cite{Szulkin}.

We would like to point out that the approach explored in \cite{Ji-Szulkin} does not work well to solve the problem $(P_1)$,  because we cannot assure that the functional
$$\Psi_2(u):=\int_{\mathbb{R}^N}F(x,u)dx$$
is a $C^1$ functional, even considering the integral on a bounded domain of $\mathbb{R}^N$. Therefore, in order to apply the Fountain Theorem to $I$, it is crucial to use the new version of the Fountain Theorem for the class $(H_0)$, see Theorem \ref{Fountain1}. Here, we show that $I$ has a sequence of critical points $(u_k)$ such that $c_k=I(u_k)\rightarrow \infty$.

In our second application, we will apply our version of the Heinz's result for functionals $I\in (H_0)$, see Theorem \ref{HT}, to prove the existence of infinitely many solutions for a concave perturbation of the logarithmic Schr\"odinger equation with constant potential, namely

$$
\left\{\begin{aligned}
-&\Delta u + u = u\log u^2 + \lambda h(x)|u|^{q-2}u  \;\;\mbox{in}\;\;\mathbb{R}^{N} \\
& u \in H^{1}(\mathbb{R}^{N}),
\end{aligned}
\right. \leqno{(P_2)}
$$
where $\lambda$ is a positive parameter, $q \in (1,2)$ and $h:\mathbb{R}^N\rightarrow \mathbb{R}$ verifies some technical conditions (see the condition $(h_1)$ below). The energy functional associated with $(P_2)$ is given by
$$I_\lambda(u):=\frac{1}{2}\int_{\mathbb{R}^N}(|\nabla u|^2+|u|^2)dx - \int_{\mathbb{R}^N}L(u)dx-\frac{\lambda}{q}\int_{\mathbb{R}^N}h(x){|u|}^{q},\,\,\,  u \in H^1(\mathbb{R}^N).$$
As in the preceding application one has $I_\lambda\in (H_0)$. Related to the above problem, we prove the existence of a sequence of critical points $(u_k)$ of $(P_2)$ with $I_\lambda(u_k)=c_k \rightarrow 0$ when $\lambda$ is small.

Finally, in our last application, we employ Theorem \ref{SP} to establish the existence of infinitely many solutions for the problem
$$
\left\{\begin{aligned}
-&\Delta_1 u = |u|^{p-2}u,  \;\;\mbox{in}\;\;\Omega, \\
& u|_{\Omega}=0, \;\; \mbox{on} \,\, \partial \Omega
\end{aligned}
\right. \leqno{(P_3)}
$$
where $\Omega \subset \mathbb{R}^N$, $N\geq 2$, is a bounded domain with smooth boundary and $p\in (1,1^*)$. The operator $\Delta_1$ is the so-called the \textit{1-Laplacian operator}, which is defined (in a formal sense) by $\displaystyle\Delta_1 u:=\text{div}\left(\frac{\nabla u}{|\nabla u|}\right)$.\par 
 The precise notion of solution for the problem $(P_3)$ will be established in the Subsection 4.3 and it follows the notion introduced by Kawohl-Schuricht in \cite{Kawohl}. Problems involving the 1-Laplacian operator have received special attention in a lot of recent studies, the interested reader can see the works \cite{Alves-Fig-Pimenta,Alves-Pimenta, Pimenta-Fig, Pimenta-Fig1, Degiovanni, Demengel, Chang2, Kawohl, Molino, Ortiz} and the references therein.

The paper is organized as follows: In Section 2 we present some preliminary notions of the nonsmooth analysis and of the group actions which will be used later on. In Section 3, we prove the news minimax theorems for functionals $I \in (H_0)$, while in Section 4 we present the applications.
\\ \\
\textbf{Notation} From now on this paper, otherwise mentioned, we fix
\begin{enumerate}
	\item[$\bullet$] $H_{rad}^1(\mathbb{R}^N):=\{u \in H^1(\mathbb{R}^N);\,\, u \,\,\text{is a radial function}\};$
	\item[$\bullet$] $C_{0,\,rad}^\infty(\mathbb{R}^N):=\{u \in C_0^\infty(\mathbb{R}^N);\,\, u \,\,\text{is a radial function} \};$
	\item[$\bullet$] $\|\cdot\|_p$ denotes the usual norm of the Lebesgue space $L^p(\mathbb{R}^N)$, $p \in [1, \infty];$
	\item[$\bullet$] If $f:\Omega\rightarrow \mathbb{R}$ is a measurable function, with $\Omega \subset \mathbb{R}^N$ a measurable set, then $\displaystyle{\int_{\Omega}} f(x)\,dx$ will be denoted by $\displaystyle{\int_{\Omega}}f\,;$
	\item[$\bullet$] $o_n(1)$ denotes a real sequence with $o_n(1) \rightarrow 0;$
	\item[$\bullet$] $C(x_1,...,x_n)$ denotes a positive constant that depends on $x_1,...,x_n;$
	\item [$\bullet$] $1^*:= \displaystyle\frac{N}{N-1}$, if $N\geq 2$;
	\item[$\bullet$] $2^*:= \displaystyle\frac{2N}{N-2}$, if $N\geq 3$ and $2^*:=\infty$ if $N=1$ or $N=2$.
	
\end{enumerate}

\section{Nonsmooth analysis, group actions and preliminary results}\label{S2}

This section concerns the notions of the nonsmooth analysis and group representations bystanders in the our main results. In what follows, $(X, \|\cdot\|)$ denotes a real Banach space and $(X^*, \|\cdot\|_*)$ its topological dual space. Moreover, let us designate by $G$ any topological group.

\subsection{A primer on Nonsmooth Analysis}
Let us first consider the locally Lipschitz case. For the proof of properties and more details see Chang \cite{Chang1}, Clarke \cite{Clarke, Clarke1}, Carl-Le-Motreanu  \cite{Motreanu1} and Motreanu-Panagiotopoulos \cite[Chapters 1-2]{Motreanu}.	
We say that a functional $F :X \rightarrow \mathbb{R} $ is locally Lipschitz ($F\in Lip_{loc}(X, \mathbb{R})$ in short), if given $u \in X$ there is an open neighborhood $V:=V_{u} \subset X$ of $u$, and a constant $K=K_{u}>0$ such that

\begin{equation*}
|F(v)-F(w)| \leq K \|v-w\|\,\;\; \forall v,w \in V.
\end{equation*}
The generalized directional derivative of $F$ at $u$ in the direction of $v \in X$ is defined by
\begin{equation*}
F^{\circ}(u;v):= \limsup_{\begin{aligned}
	&h\rightarrow 0,\\
	&\delta \rightarrow 0^+
	\end{aligned}}\frac{1}{\delta}\left(F(u+h+\delta v)-F(u+h)\right)
\end{equation*}
and the generalized gradient of $F$ at $u$ as being the set
\begin{equation*}
\partial F(u)= \{\phi \in X^{*}: F^{\circ}(u;v) \geq \left<\phi, v\right>,\,\forall\; v \in X \}.
\end{equation*}
In the sequel, a point $u \in X$ is a critical point of $F$ if $0 \in \partial F(u)$.

For $F\in Lip_{loc}(X, \mathbb{R})$, it holds

$$F^{\circ}(u;v):=\max\{\langle \eta, v\rangle:\, \eta \in \partial F(u)\}.$$

We recall that, when a functional $Q:X\rightarrow\mathbb{R}$ is convex, the subdifferential of $Q$ at $u$ is the set
\label{sd}\begin{equation}
\partial_s Q(u):=\{\phi \in X^*: \, Q(v)-Q(u)\geq \left<\phi,v-u\right>, \forall\,v \in X\}.
\end{equation}
If  $Q$ belongs to $Lip_{loc}(X, \mathbb{R})$, then $\partial_s Q(u)=\partial Q(u)$.

We list some usual properties of the generalized directional derivative and of the generalized gradient.

\begin{lemma}\label{usc} Let $F \in Lip_{loc}(X,\mathbb{R}),$ then
\begin{itemize}
	\item[${i)}$] the map $(u,v) \mapsto F^\circ(u,v)$ is an upper semicontinuous functional, i.e. if $(u_j,v_j) \rightarrow (u,v)$ then
	$$ \limsup F^\circ(u_j,v_j) \leq F^\circ(u,v);$$
	\item[${ii)}$] $F^\circ(u,-v)=(-F)^\circ(u,v).$
\end{itemize}
\end{lemma}

\begin{lemma}
	If $F$ is continuously Fréchet differentiable in an open neighborhood of $u \in X$, we have $\partial F(u)= \{F'(u)\}$.
\end{lemma}

\begin{lemma}\label{S}
	If $F,\,Q \in Lip_{loc}(X, \mathbb{R})$, then for each $u \in X$ one has
\begin{itemize}
	\item [$i)$]$\partial(F+Q)(u) \subset \partial F(u) + \partial Q(u);$
	\item [$ii)$]$\partial(F+Q)(u) = \{F'(u)\} + \partial Q(u)$, if $F \in C^1(X,\mathbb{R})$.
\end{itemize}
\end{lemma}

Next, we recall a lemma that establishes an important property between $F^\circ(u,v)$ and the classical directional derivatives of $F$ at $u\in X$ along $v\in X,$ i.e.,
	\begin{equation}
	\displaystyle{\frac{\partial F}{\partial v}}(u):= \lim_{t\rightarrow 0^+}\frac{F(u+tv)-F(u)}{t}.
	\end{equation}

\begin{lemma}\label{G}
	If $F:X\rightarrow\mathbb{R}$ is a convex function and $F \in Lip_{loc}(X,\mathbb{R})$, then $\displaystyle{\frac{\partial F}{\partial v}}(u)$ exists for any $u, v \in X$ and $\displaystyle{\frac{\partial F}{\partial v}}(u)=F^\circ(u,v)$.
\end{lemma}

From now on, we say that $I:X\rightarrow (-\infty, \infty]$ belongs to the class $(H_0)$ if $I$ satisfies:
\\ \\
$I:=\Phi+\Psi$, with $\Phi \in C^1(X,\mathbb{R})$ and $\Psi:X\rightarrow (-\infty,\infty]$ is a convex lower semicontinuous functional and does not satisfy $\Psi \equiv \infty$.
\\\\
In particular, $I$ is a lower semicontinuous (l.s.c.) functional.

The \textit{effective domain} of $I$ is defined by
\begin{equation*}
D(I):=\{u \in X: \,I(u) < \infty\},
\end{equation*}
and so, for $I \in (H_0)$ it is true that $D(I)=D(\Psi)$.
For each $u \in D(I)$, we say that the subdifferential of $I$ at $u$ is the set
\begin{equation}
\partial I(u) := \{\varphi \in X^*:\,\,\left<\Phi'(u), v-u\right> + \Psi(v)-\Psi(u) \geq \left<\varphi,v-u\right>,\, \forall v \in X \}.
\end{equation}	

\begin{definition}\label{cp}
	Suppose that $I\in (H_0)$ with $I:=\Phi + \Psi$. Then
\begin{itemize}
	\item[$i)$] a point $u \in X$ is called a critical point of $I$ if $0 \in \partial I(u),$ or more precisely, $u \in D(I)$ and
	\begin{equation*}
	\left<\Phi'(u), v-u\right> + \Psi(v)-\Psi(u) \geq 0\,\, \,\,\forall v \in X,
	\end{equation*}
	\item[$ii)$] a sequence $(u_n)$ is called a Palais-Smale sequence $($briefly $(\rm PS)$ sequence$)$ for $I$ at level $c \in \mathbb{R}$ if $I(u_n) \rightarrow c$ and
	\begin{equation*}
	\left<\Phi'(u_n), v-u_n\right> + \Psi(v)-\Psi(u_n) \geq -\varepsilon_n\|v-u_n\|\,\,\,\, \forall v \in X,
	\end{equation*}
	with $\varepsilon_n\rightarrow 0^+$, or equivalently (see, \cite[Proposition 1.2]{Szulkin}) 
	\begin{equation*}
	\left<\Phi'(u_n), v-u_n\right> + \Psi(v)-\Psi(u_n) \geq \langle w_n,v-u_n \rangle \,\,\,\, \forall v \in X,
	\end{equation*}
	where $w_n \in X^*$ with $w_n \rightarrow 0$ in $X^*;$
	\item[$iii)$] $I$ satisfies the Palais-Smale condition $($briefly $(\rm PS)$ condition$)$ at level $c\in \mathbb{R}$ when each $(\rm PS)$ sequence $(u_n)$ at level $c$ has a convergent subsequence. If $I$ verifies the $(\rm PS)$ condition for all level $c$, we say simply that $I$ satisfies the $(\rm PS)$ condition.
\end{itemize}
\end{definition}
For further details about the critical point theory described above see Szulkin \cite{Szulkin}, Carl-Le-Motreanu \cite{Motreanu1} and  Motreanu-Panagiotopoulos \cite{Motreanu}.

Now, let us recall the classical Ekeland's Variational Principle \cite[Theorem 1]{Ekeland} that will be useful in the sequel.

\begin{theorem}\label{PVE}
	Let $(Y,d)$ be a complete metric space. Suppose that $\varphi:Y\rightarrow(-\infty,\infty]$ is a proper l.s.c. functional bounded below. Given $\delta$, $\tau>0$ and $u_0 \in Y$ such that
	\begin{equation}
	\displaystyle{\inf_{u\in Y}}\varphi(u) \leq \varphi(u_0)\leq\displaystyle{\inf_{u\in Y}}\varphi(u)+\delta,
	\end{equation}
	then there is $v_0 \in Y$ verifying
\begin{itemize}
	\item [$i)$]$\varphi(v_0)\leq\varphi(u_0),\,\,\,d(v_0,u_0)\leq {1}/{\tau};$
	\item [$ii)$]$\varphi(v)-\varphi(v_0)\geq-\delta\tau d(v,v_0),\, \forall v \in Y$.
\end{itemize}
\end{theorem}

\subsection{Group Actions.}

The notions described in this subsection follow closely the presentation in \cite[Sections 1.6 and 3.2]{Willem}; see also Bartsch \cite{Bartsch} for additional comments and remarks. Let $G$ be a topological group with neutral element $e$. An action of $G$ on $X$ is a continuous function
\begin{equation*}
\begin{aligned}
\phi: \,&G\times X&\rightarrow &\,X\\
\,&(g,v)&\mapsto\,&\phi(g,v)=gv
\end{aligned}
\end{equation*}
such that 
\begin{itemize}
\item[$(G_1)$] $ev = v,\,\, \forall x \in X$;
\item[$(G_2)$] $(gh)v=g(hv),\,\, \forall v \in X,\, \forall g,h \in G$;
\item[$(G_3)$]  \it{For each $g \in G$ the map
\begin{equation*}
\begin{aligned}
\phi_g:\,&X\,&\rightarrow\,\, &\,X \\
\,&v &\mapsto\,\, &\phi_g(v)=gv
\end{aligned}
\end{equation*}
is linear.}
\end{itemize}
\indent If in addition to the above condition, the following relation holds 
\begin{itemize}
\item[$(G_4)$] $\|gv\|=\|v\|,\,\, \forall v \in X,\,\,\forall g \in G $
\end{itemize}
the map $\phi$ is said to be an isometric action. According to the above definitions, we say that $G$ acts isometrically on $X$ when $(G_1)-(G_4)$ hold.

The subspace of invariant elements of $X$ is defined by
$$Fix(G):= \{u\in X:\, gu=u\,\,\,\forall g \in G\}.$$
\begin{example}\label{ex1}
	If $Id$ denotes the identity map on $X$ and considering the representation \linebreak $\mathbb{Z}_2 = \{Id, -Id\}$, it is standard to check that $\mathbb{Z}_2$ acts isometrically on $X$.
\end{example}

A subset $A$ of $X$ is said to be \textit{$G$-invariant} if $gA=A$ for every $g \in G$, where $gA:= \{gx:\,\,x\in A\}$. Also, when $A \subset X$ is a $G$-invariant set, a map $\gamma:A\rightarrow X$ is called \textit{equivariant map} if
$$ \gamma(gx)=g\gamma(x)\,\,\,\, \forall x \in A,\, \forall g \in G.$$
If a functional (not necessarily linear) $\varphi$ defined on $X$ satisfies $\varphi(gx)=\varphi(x)$ for any $x \in X$ and $g \in G$, we say that $\varphi$ is a $G$-invariant functional.
\\
\\
\textbf{Notation}: $\Gamma_G(A):= \{\gamma \in C(A,X):\, \gamma \, \text{is equivariant}\}.$
\\

By following \cite[Section 3.2]{Willem} and \cite{Bartsch1}, the notion of admissible action is given below.

\begin{definition}
	Let $Y$ be a finite dimensional vector space. Moreover, let us assume that $G$ is a compact topological group that acts diagonally on $Y^k,$ that is
	$$ gv = g(v_1,...,v_k) = (gv_1,...,gv_k), $$
for every $v=(v_1,...,v_k) \in Y^k$ and each $ g \in G$.
	The action of  $G$ on $Y$ is said to be admissible if, for each equivariant map $\gamma:\partial U \rightarrow Y^{k-1},$ where $k\geq 2$ and $U$ is a bounded $G$-invariant open set of $Y^k$ with $0 \in U,$ there is $u \in \partial U$ such that $\gamma(u) = 0$.
\end{definition}

For our goals we will consider a special condition on a decomposition of space $X$ with respect to action of $G$ on $X$ as follows:
\begin{itemize}
\item[$(G_0)$] \it{ $G$ is a compact group that acts isometrically on $$X=\overline{\displaystyle{\bigoplus_{j \in \mathbb{N}}} X_j},$$ where every $X_j$ is a $G$-invariant subspace of $X$ such that $X_j \cong Y,$ being $Y$ a finite dimensional vector space for which the action of $G$ is admissible.}
\end{itemize}

We also emphasize that a key point in our approach is given by the \textit{Haar integral} on a topological group $G$. For the sake of completeness a brief description of this abstract concept will be given below; see Nachbin \cite{Nachbin} for additional comments and details. Suppose that $G$ is a locally compact group and $\mu$ a positive measure on $G$. According to the classical literature on the subject, $\mathcal{L}(G,\mu)$ denotes here the space of the integrable functions on the group $G$ with respect to the measure $\mu$. We say that $\mu$ is a \textit{left invariant} if 
\begin{equation}\label{Haar4}
\int_G f(g^{-1}y) d\mu = \int_G f(y)d\mu,\,\, \forall g \in G,
\end{equation}
for every $f \in \mathcal{L}(G,\mu)$. Here, we would like to mention that the integral above is a version of the Bochner integral to the Haar integral.  By repeating an analogous procedure as made in the building of Bochner integral (see e.g. \cite[Apenddix E and refrences therein]{Evans}), the reader can note that it is possible, by using the measure $\mu$, to extend the Haar integral for functions $f:G\longmapsto X$. We notice this version of the Haar integral also satisfies the property in (\ref{Haar4}).

The next result assures the existence of a left invariant measure on a locally compact topological group $G$.

\begin{theorem}[Haar]\label{Haar}
	Let $G$ be a locally compact group. Then, there exists at least one left invariant positive measure $\mu_0 \neq 0$. Moreover, the measure $\mu_0$ is unique except for a strictly positive factor of proportionality, i.e. if $\mu_1$ is a left invariant positive measure on $G,$ there exists $c>0$ such that $\mu_1=c\mu_0$. Finally
$$\mu_0(G) < \infty \Leftrightarrow \mbox{ $G$ is compact.}$$ 
\end{theorem}

	See \cite[Chapter II, Sections 4 and 5]{Nachbin} for a detailed proof.	
\begin{corollary}[Normalized Haar measure]\label{Haar2}
	Let $G$ be a compact group. Then, there exists a left invariant positive measure $\mu$ on $G$ such that $\mu(G)=1$.
\end{corollary}
\begin{proof}
	Take $\mu:=\displaystyle{\frac{1}{\mu_0(G)}}\mu_0$, with $\mu_0$ given in the Theorem \ref{Haar}.
\end{proof}	

	\begin{remark}\rm{The integral associated to $\mu_0$ in the Theorem \ref{Haar} is the so called {Haar's integral}. From now on, we will denote by $\mu$ the measure with the property given in Corollary \ref{Haar2}.}
	\end{remark}

Let $\beta:X \rightarrow X$ be a continuous map and let $G$ be a compact topological group.  By the left invariance property of $\mu$, if $\eta:X\rightarrow X$ is the map given by
\begin{equation}\label{Haar3}
\eta(u):= \int_Gg\beta(g^{-1}u)d\mu,\,\, u\in X,
\end{equation}
then $\eta \in \Gamma_G(X)$. This fact will be useful in the sequel.

We finish this subsection by recalling an important result due to Kobayashi-Ôtani which generalizes the Principle of symmetric criticality due to Palais; see \cite[Theorem 1.28]{Willem}.

\begin{theorem}\label{PC}
	Let $X$ be a reflexive Banach space and let $G$ be a compact topological group that acts isometrically on $X$. If $I \in (H_0)$ with $\Phi$ and $\Psi$ being $G$-invariant, then
	\begin{equation}\label{PC1}
	0 \in \partial (I|_Z)(u) \implies 0 \in \partial I(u),
	\end{equation}
for any $u \in Z:=Fix(G)$.
\end{theorem}
An exhaustive proof of Theorem \ref{PC} is given in \cite[Theorem 3.16]{Kobayashi}. We emphasize here that (\ref{PC1}) ensures that every critical point of $I|_Z$ is a critical point of $I$ in the sense of Definition \ref{cp}.

\subsection{$G$-index theory}

Now, we introduce the notion of the $G$-index that will be required in our abstract results. The reader can consult \cite{Bartsch0} for a discussion in a more general situation.  
Let $\Sigma$ be the class of subsets of $(X-\{0\})$ that are $G$-invariant and closed in $X$. let us assume that the condition $(G_0)$ holds and let $Y$ be the vector space fixed in that condition.

\begin{definition}
	The $G$-index of $A \in \Sigma\setminus\{\emptyset\}$ is defined as 
$$
\gamma_G(A):=\min\{k\in \mathbb{N}\setminus\{0\}: \exists \phi: A \rightarrow Y^k\setminus\{0\},\,\phi \in \Gamma_G(A) \}
$$
if such integer exists and $\gamma_G(A):=\infty$ otherwise. Finally, we also set $\gamma_G(\emptyset):=0$.
\end{definition}

\begin{remark}\label{GT}
\rm{Note that when $G=\mathds{Z}_2$ the $G$-index introduced above coincides with the genus of symmetric subset of $(X-\{0\})$ (see \cite{Rabinowitz} for more details on genus theory).}
\end{remark}	

Denote by $\mathcal{C}$ the collection of all nonempty closed and bounded subsets of $X$. In $\mathcal{C}$ we put the Hausdorff metric $d_H$ given by
$$ d_H(A,B):=\max\left\{\sup_{x \in A} d(x,B),\,\sup_{y\in B} d(y,A)\right\},\,\,\, A,B \in \mathcal{C},$$
where $d$ denotes the usual distance on $X$. It is well known that $(\mathcal{C},d_H)$ is a complete metric space. Denote by $\mathcal{D}_G$ the subcollection of $\mathcal{C}$ of all nonempty compact $G$-invariant subset of $X$. By following the ideas in \cite[Section 4]{Szulkin} the reader is invited to note that $(\mathcal{D}_G,d_H)$ is a complete metric space (replace the condition of symmetry in that reference by the notion of $G$-invariance). By a similar way, we notice that, setting
$$ \Gamma_j:=\overline{\{A \in \mathcal{D}_G;\,0\notin A,\, \gamma_G(A)\geq j\}}^{d_H},$$
the space $(\Gamma_j,d_H)$, is also a complete metric space. The next properties can be proved by using an analogous reasoning as made in  \cite{Rabinowitz}.

\begin{proposition}\label{GP}
	For every $A,B \in \Sigma$ the following facts hold:
\begin{itemize} 
	\item[$i)$] If there exists $\phi:A \rightarrow B$, $\phi \in\Gamma_G(A)$, then $\gamma(A)\leq \gamma(B);$
	\item[$ii)$] $A \subset B$ implies that $\gamma(A)\leq\gamma(B);$
	\item[$iii)$] $\gamma(A\cup B)\leq \gamma(A)+\gamma(B);$
	\item[$iv)$] $\gamma(\overline{A\setminus B})\geq\gamma(A)-\gamma(B),$ since $\gamma(B)<\infty;$
	\item[$v)$] If $G$ is a finite group and $A$ is a compact set, then $\gamma(A)<\infty$.
	\item[$vi)$] If $A$ is a compact set, then we have $\gamma(N_\delta(A))=\gamma(A),$ $\delta \approx 0^+,$ where $$N_\delta(A):=\{x \in X:\, d(x,A)\leq \delta\}.$$
\end{itemize} 
\end{proposition}

Finally, by following the idea in \cite[Proposition 4.2]{Szulkin}, we can prove the property below.

\begin{proposition}\label{GP1}
	If $A\in \Gamma_j$ is such that $0 \notin A$, then $\gamma(A)\geq j$.
\end{proposition}	

\section{Deformation Lemmas and minimax theorems for lower semicontinuous functionals}

In order to prove the main variant (Theorem \ref{Fountain1} below) of the classical Fountain Theorem,
at the beginning of this section we recall 
a suitable version, valid for the class of functionals $I \in (H_0)$, of 
the deformation lemma proved by
Szulkin in \cite[Proposition 2.3]{Szulkin}. In addition, we show an equivariant version of this result.
Finally, in the last subsection  we prove two more abstract results: the result proved by Heinz in \cite[Proposition 2.2]{Heinz} 
has been extended to the class of functionals $I\in (H_0)$ as well as  a version of a result due to Szulkin in \cite[Corollary 4.8]{Szulkin}. In our case, the critical levels $c_k$ obtained in the result satisfy $c_k\rightarrow \infty$. The main results of this section complement the study made by Szulkin in \cite{Szulkin}, since new minimax results are established to the class $(H_0)$.

For $I:=\Phi+\Psi \in (H_0)$, let us denote by $I^{c}$, $K$ and $K_c$ the following sets
$$I^{c}:=I^{-1}((-\infty, c])\, \mbox{ for every }\, c \in \mathbb{R},$$
and
$$K:=\{u \in X: \, u \,\, \text{is a critical point of}\,\, I\},$$
as well as
$$K_c:= \{u \in K:\, I(u)=c\}.$$
\subsection{Deformation lemmas and Fountain Theorem}
We recall that by a \textit{deformation} we mean a family of mappings of type
\begin{equation*}
\alpha_s:=\alpha(s,\cdot):W\subset X \rightarrow X, \,\,s \in [0, s_0]
\end{equation*}
such that $\alpha_0 \equiv Id|_W$, with $\alpha \in C([0, s_0]\times W,X)$ and $Id|_W$ is the identity map restricted to $W$.

Next, we recall the deformation lemma proved by Szulkin in \cite[Proposition 2.3]{Szulkin}.

\begin{lemma}\label{DL}
	Suppose that $I=\Phi + \Psi \in (H_0)$ satisfies the $(\rm PS)$ condition and let $N$ be a neighbourhood of $K_c$. Then, fixed $\varepsilon_0 > 0$, there is $\varepsilon \in (0,\varepsilon_0)$ such that, for each compact set $A \subset X\setminus N$ with
	\begin{equation*}
	c \leq \displaystyle{\sup_{u \in A}} I(u) \leq c+\varepsilon,
	\end{equation*}
there exist a closed set $W,$ with $A\subset \text{int}({W}),$ and a deformation $\alpha_s: W \rightarrow X,$ with $0 \leq s \leq s_0 \approx 0^+,$ such that
\begin{itemize}
	\item[$i)$] $ \|\alpha_s(u) - u \| \leq s,\,\,\, \forall u \in W;$
	\item[$ii)$] There is a number $\delta = \delta_\varepsilon \approx0^+$ such that
	\begin{equation*}
	I(\alpha_s(u)) - I(u) \leq s + \delta s\,\,\,\,\, \forall u \in W,
	\end{equation*}
	and
	\begin{equation*}
	I(\alpha_s(u)) - I(u) \leq -3\varepsilon s + \delta s\,\,\,\,\, \forall u \in W,\,\, I(u) \geq c - \varepsilon.
	\end{equation*}
\end{itemize}
	Moreover, by $ii)$ it follows that
\begin{itemize}
	\item[$iii)$] $I(\alpha_s(u))-I(u) \leq 2s,\,\,\, \forall u \in W;$
	\item[$iv)$] $I(\alpha_s(u)) - I(u) \leq -2\varepsilon s, \,\,\, \forall u \in W,\,\, I(u)\geq c-\varepsilon;$
	\item[$v)$] $\displaystyle{\sup_{u \in A}}I(\alpha_s(u)) - \displaystyle{\sup_{u \in A}} I(u)\leq -2\varepsilon s.$
	\item[$vi)$] $I(\alpha_s(u))-I(u)\leq 0,\,\,\, \forall u \in W\cap C,$ for each closed set verifying $C\cap K=\emptyset$.
\end{itemize}
\end{lemma}

	 We would like to point out that $ii)$ is not contained in the statement of \cite[Proposition 2.3]{Szulkin}. However, the sufficiently small $\delta>0$ in $ii)$ appears in the proof of \cite[Proposition 2.3]{Szulkin}.
\smallskip

Now, we are able to prove an equivariant version of Lemma \ref{DL} making use of the next concept involving the action of a compact topological group $G$ on $X$ and a functional of type $\Psi: X \rightarrow (-\infty,\infty]$.

\begin{definition}
	Let $\Psi:X\rightarrow (-\infty, \infty]$ be a functional and $G$ be a compact topological group that acts on $X$. We say that $\Psi$ is compatible with the action of $G$ on $X$ ($G$-compatible for short) if
	\begin{equation}\label{Comp}
	\Psi\left(\int_G \beta(g u) d\mu\right) \leq \int_G\Psi(\beta(g u))d\mu,
	\end{equation}
	for every fixed $u\in X$, $\beta \in C(Gu,X),$ where $Gu:=\{gu;\,g\in G\}$ and $\mu$ denotes the normalized Haar measure.
\end{definition}

The inequality in (\ref{Comp}) is verified in some meaningful cases and some of them are briefly discussed in the next example.

\begin{example} \rm{By using the usual notations, let us restrict our attention to the following cases:\par
\begin{itemize}
	\item[$1)$] Let $\Psi\equiv\|\cdot\|:X \rightarrow \mathbb{R}$ be the norm defined on $X$. Fixed a map $\beta \in C(G,X),$ let $(\beta_n)$ be a sequence of simple functions with
	\begin{equation}\label{Conv}
	\int_G \beta_n(g)d\mu \rightarrow \int_G \beta(g) d\mu \quad \mbox{and}\quad \int_G \|\beta_n(g)\|d\mu \rightarrow \int_G \|\beta(g)\|d\mu.
	\end{equation}
	Each function $\beta_n$ can be written as a finite sum:
	\begin{equation*}
	\beta_n = \sum_i \chi_{A_i}v_i\quad \mbox{where}\quad A_i:= \beta_n^{-1}(\{v_i\})\quad\mbox{and}\quad v_i \in X.
	\end{equation*}
	Since $\mu$ is the normalized Haar measure ($\mu (G) = 1$), we have $\displaystyle{\sum_i} \mu(A_i) = 1$ and 
	\begin{equation*}
	\left\Arrowvert\int_G \beta_n(g)d\mu\right\Arrowvert = \left\Arrowvert \sum_i \mu(A_i) v_i\right\Arrowvert \leq \sum_i \mu(A_i)\|v_i\|=\int_G\|\beta_n(g)\|d\mu,
	\end{equation*}
for every $n \in \mathbb{N}$. 
	Consequently, by using (\ref{Conv}) it follows that
	\begin{equation*}
	\left\Arrowvert\int_G \beta(g)d\mu\right\Arrowvert \leq \int_G\|\beta(g)\|d\mu.
	\end{equation*}
	So $\|\cdot\|$ is compatible with the action of $G$ on $X$. In general, the result is still true for an arbitrary convex continuous function $\Psi:X\rightarrow \mathbb{R}$.
	\item[$2)$] Let us assume that $G:=\{g_1,..., g_k\}$ is a finite group. Since 
$$\displaystyle{\sum_{i=1}^{k}} \mu(\{g_i\}) = 1,$$ 
{for each} $\beta \in C(G,X)$ the integral $\displaystyle{\int_G} \beta(g)d\mu$ can be written as a finite convex combination of vectors of  $X$. More precisely, one has
	\begin{equation*}
	\int_G \beta(g)d\mu = \sum_{i=1}^{k} \mu(\{g_i\})v_i,
	\end{equation*}
where $v_i := \beta(g_i)$.\par
	\noindent Then, for any convex l.s.c. functional $\Psi:X \rightarrow (-\infty,\infty]$ one has
	\begin{equation*}
	\Psi\left(\int_G \beta(g)d\mu\right) = \Psi\left(\sum_{i=1}^{k} \mu(\{g_i\})v_i\right)\leq \sum_{i=1}^{k} \mu(\{g_i\})\Psi(v_i)=\int_G\Psi(\beta(g))d\mu,
	\end{equation*}
	i.e. $\Psi$ is compatible with the action of $G$ on $X$.
\end{itemize}
}
\end{example}

The next result (Equivariant Deformation Lemma) generalizes the Corollary 2.4 proved by Szulkin in \cite{Szulkin} and also complements the Lemma 5.1 found in Bereanu-Jebelean \cite{BJ}. 

\begin{lemma}\label{EDL}
	Suppose that $I:=\Phi+\Psi \in (H_0)$ satisfies the $(\rm PS)$ condition, where $\Phi$ and $\Psi$ are $G$-invariant functionals and $\Psi$ is compatible with the action of the compact topological group $G$ on $X$. Assume that $G$ acts isometrically on $X$. Under the hypothesis of Lemma \ref{DL}, the same conclusions hold with $\alpha_s: W \rightarrow X$  equivariant in $A$, whenever $A$ is a $G$-invariant set.
\end{lemma}	
\begin{proof}
	Denote by $\beta_s$ the deformation of Lemma \ref{DL} and set
	\begin{equation}
	\alpha_s(u):=\int_Gg^{-1}\beta_s(gu)d\mu.
	\end{equation}
	Thanks to (\ref{Haar3}), we observe that $\alpha_s \in \Gamma_G(A)$ and let us prove that $\alpha_s$ verifies all the assumptions of Lemma \ref{DL}. More precisely, since $iii),$ $iv)$ and $v)$ are a direct consequence of $ii)$, and so, it is enough to show only $i)$ and $ii)$. By Lemma \ref{DL} - Part $i)$, one has
	\begin{equation}\label{i}
	\begin{aligned}
	\left\Arrowvert \alpha_s(u)-u\right\Arrowvert&=\left\Arrowvert \int_Gg^{-1}\beta_s(gu)d\mu - \int_G(g^{-1}g)ud\mu\right\Arrowvert\\
	&\leq \int_G\left \Arrowvert g^{-1}(\beta_s(gu)-gu)\right\Arrowvert d\mu\\
	& \leq \int_G s d\mu\, =s\quad \mbox{for every}\,\, u \in W,
	\end{aligned}
	\end{equation}
	i.e. $\alpha_s$ verifies $i)$ as claimed.\par 
\noindent In order to prove $ii)$  let us write $\beta_s(u)=u+h_s(u),$ so that $\alpha_s(u)= u + w_s(u)$, where $w_s(u)=\displaystyle{\int_G} g^{-1}h_s(gu)d\mu$. Consequently, the Taylor's formula immediately yields
	\begin{equation}\label{Taylor}
	I(\alpha_s(u))= \{\Phi(u)+ \langle \Phi'(u), w_s(u)\rangle + r(s)\} + \Psi(\alpha_s(u)),\quad \frac{r(s)}{s}=o_s(1).
	\end{equation}
	Now, the compatibility condition of $\Psi$ gives 
	\begin{equation}
	\begin{aligned}
	I(\alpha_s(u))\leq\int_G(\Phi(u)+ \langle \Phi'(u),g^{-1}h_s(gu)\rangle)d\mu+\int_G\Psi(g^{-1}\beta_s(gu))d\mu + \frac{\delta}{2}s,		
	\end{aligned}
	\end{equation}
for $s \approx 0^+$. Moreover, since 
$$\langle\Phi'(u), g^{-1}h_s(gu)\rangle = \langle\Phi'(gu),h_s(gu)\rangle,$$ the $G$-invariance of $\Phi$ and the Taylor's expansion applied to $I(\beta_s(gu))$ give
	\begin{equation}\label{Taylor2}
	\begin{aligned}
	I(\alpha_s(u))&\leq\int_G(\Phi(gu)+ \langle \Phi'(gu),h_s(gu)\rangle)d\mu+\int_G\Psi(\beta_s(gu))d\mu + \frac{\delta}{2}s\\
	&=\int_G (I(\beta_s(gu))-\rho(s))d\mu + \frac{\delta}{2}s\leq \int_G I(\beta_s(gu))d\mu +\delta s.	
	\end{aligned}
	\end{equation}
	(Here, we have used $\rho$ as being the rest in the Taylor's expansion). Finally, by Lemma \ref{DL} - Part $ii)$ and (\ref{Taylor2}), it follows that
	\begin{equation}\label{e1}
	I(\alpha_s(u))\leq \int_G I(gu)d\mu + s+2\delta s \leq I(u)+s+2\delta s,
	\end{equation}
for every $u\in W$.
	Similarly 
	\begin{equation}\label{e2}
	I(\alpha_s(u))\leq I(u) - 3\varepsilon s + 2\delta s,\quad\mbox{for every}\quad u \in W\,\,\mbox{ and }\,\,\,I(u)\geq c-\varepsilon.
	\end{equation}
	 Inequalities (\ref{e1}) and (\ref{e2}) ensure that $\alpha_s$ satisfies $ii)$ provided that $\delta$ is sufficiently small.
\end{proof}	

For the sake of completeness, let us recall now the notion of \textit{homotopy}. Let $B$ be a subset of $X$ and $f$, $g \in C(B,X)$. We say that $f$ is homotopic to $g$ if there is $h \in C([0,1]\times B,X)$ satisfying
\begin{equation}\label{Homotopy}
h(0,\cdot)\equiv f \,\,\,\, \text{and} \,\,\,\, h(1,\cdot)\equiv g.
\end{equation}
The map $h$ is called a homotopy between $f$ and $g$. We denote $f \approx g$ to designate that $f$ is homotopic to $g$ by an equivariant homotopy, i.e., there exists $h \in C([0,1]\times B,X)$ satisfying (\ref{Homotopy}) with $h(t,\cdot) \in \Gamma_G(B)$ for any $t \in [0,1]$. It easily seen that $\approx$ is an equivalence relation in $C(B,X)$.

In what follows, for each $k \in \mathbb{N}$, we set
\begin{itemize}
\item[$i)$] $Y_k := \bigoplus_{j=1}^k X_j$ and $Z_k:= \overline{\bigoplus_{j=k}^\infty X_j};$
\item[$ii)$] $B_k:=\{u \in Y_k;\,\, \|u\|\leq \rho_k\}$ and $N_k:=\{u \in Z_k;\,\, \|u\|=r_k\},$ with $\rho_k>r_k>0.$
\end{itemize}

Finally, let us recall the Intersection Lemma proved in \cite[Lemma 3.4]{Willem}; see also \cite[Theorem 2]{Bartsch1} for additional comments and remarks.

\begin{lemma}\label{Intersection}
	Assume that $(G_0)$ holds. If $\gamma \in C(B_k,X)\cap\Gamma_G(B_k)$ and $\gamma|{\partial B_k}\equiv Id|_{\partial B_k}$, then $\gamma(B_k)\cap N_k \neq \emptyset$.
\end{lemma}

Now, we are ready to show a version of the classical Fountain Theorem due to Bartsch \cite{Bartsch0} that is valid for the functionals $I\in (H_0)$.

\begin{theorem}\label{Fountain1}
	Suppose $I:=\Phi+\Psi \in (H_0)$ satisfies the $(\rm PS)$ condition, where $\Phi$ and $\Psi$ are $G$-invariants functionals with $\Psi$ compatible with the action of $G$ on $X$. Moreover, assume that $(G_0)$ holds and that
	\begin{itemize}
		\item [$i)$] $a_k=\displaystyle{\sup_{u \in Y_k, \|u\|= \rho_k}}I(u)\leq 0;$
		\item [$ii)$] $	b_k:= \displaystyle{\inf_{u \in Z_k, \|u\|=r_k}}I(u) \rightarrow \infty.$
	\end{itemize}
for every $k\geq2$.\par
	Then, supposing that the levels $c_k:=\displaystyle{\inf_{\gamma \in \Theta_k}}\,\displaystyle{\sup_{u\in B_k}} I(\gamma(u)) <\infty,$ where
\begin{equation}\label{Collection}
\,\,\,	\Theta_k:=\{\gamma \in C(B_k,X):\,\, \gamma \in \Gamma_G(B_k) \;\; \mbox{and} \;\; \gamma|_{\partial B_k}\equiv Id|_{\partial B_k}\},
\end{equation}
	the functional $I$ has infinitely many critical points $\{u_k\}$ such that $I(u_k)=c_k\rightarrow \infty$.
\end{theorem}
\begin{proof}
	Arguing by contradiction, if the theorem is not true, we may assume that for some $k\geq 2$ we have $K_{c_k}=\emptyset$. If $k$ is large enough, from Lemma \ref{Intersection}, we have $c_k \geq b_k>0$, and so, we can apply the Lemma \ref{DL} with $N=\emptyset$ and $\varepsilon_0=c_k$. Fixing $\varepsilon \in (0,c_k)$ given in Lemma  \ref{DL}, we will get  a contradiction. Indeed, let us define
	\begin{equation}\label{Collec}
	\tilde{\Theta}_k:=\left\{\gamma \in C(B_k,X):\,\, \gamma \in \Gamma_G(B_k),\,\,\gamma|_{\partial B_k}\approx Id|_{\partial B_k}\,\, \text{in}\,\,I^{c_k-\frac{\varepsilon}{4}}\,\, \mbox{and} \,\,  (I\circ\gamma)|_{\partial B_k}\leq \left(c_k-\frac{\varepsilon}{2}\right)\right\}.
	\end{equation}
	Thanks to conditions $i)-ii)$ above, if $\gamma \in \Theta_k$ and $u\in \partial B_k$, we derive
	$$I(\gamma(u))=I(u)\leq 0 <c_k-\frac{\varepsilon}{2}<c_k-\frac{\varepsilon}{4}.$$ Hence $\Theta_k \subset \tilde{\Theta}_k$ and
	\begin{equation}\label{Inf}
	\tilde{c}_k:=\displaystyle{\inf_{\gamma \in \tilde{\Theta}_k}}\,\displaystyle{\sup_{u\in B_k}} I(\gamma(u)) \leq c_k.
	\end{equation}
	
	\noindent If  $\tilde{c}_k < c_k$, it easily seen that there exists $\gamma_0 \in \tilde{\Theta}_k$ such that $$m_0:=\displaystyle{\sup_{u\in B_k}}I(\gamma_0(u))<c_k.$$ Moreover, by (\ref{Collec}), there exists a homotopy $H\in C\left([0,1]\times \partial B_k,I^{c_k-\frac{\varepsilon}{4}}\right)$ such that
	\begin{equation}\label{H1}
	H(0,\cdot)\equiv \gamma_0|_{\partial B_k}\,\,\, \text{and}\,\,\, H(1,\cdot)\equiv Id|_{\partial B_k},
	\end{equation}
	with $H(t,\cdot)$ equivariant for every $t\in[0,1]$. Since $B_k$ is a ball of radius $\rho_k$ each point $u \in B_k$ can be represented as
	$u \equiv (s, \tilde{u})$, $s \in [0, \rho_k]$, $\tilde{u} \in \partial B_k$; polar coordinates of $u$. Hence, if $u \in \partial B_k$ then $u\equiv(\rho_k,u)$. Now, define $\gamma_1:B_K \to X$ by 
	\begin{equation}
	\gamma_1(s,v) := \left\{\begin{aligned}&\gamma_0(s,v)\,\,\,\,\,\, s\in \left[0,\,\frac{\rho_k}{2}\right]\\
	&H\left(\frac{2}{\rho_k}s-1,v\right)\,\,\,\,\,\, s\in\left[\frac{\rho_k}{2},\,\rho_k\right].
	\end{aligned}\right.
	\end{equation}

According to (\ref{H1}), when $s={\rho_k}/{2}$ it holds $H({2s}/{\rho_k}-1,\cdot)=H(0,\cdot)\equiv\gamma_0$, which assures that $\gamma_1$ is well defined and $\gamma_1 \in C(B_k,X)\cap \Gamma_G(B_k)$, since $\gamma_0$ and $H(t,\cdot)$ are equivariants. By using again (\ref{H1}), if $u\in\partial B_k$ one has
	\begin{equation*}
	\gamma_1(u)=H(1,u)=Id|_{\partial B_k}(u),
	\end{equation*}
so that $\gamma_1 \in \Theta_k$. But $$\displaystyle{\sup_{u\in B_k}} I(\gamma_1(u)) \leq \max\left\{m_0, c_k-\frac{\varepsilon}{4}\right\}<c_k,$$ against the definition of $c_k$. This contradiction assures that $\tilde{c}_k=c_k$ in (\ref{Inf}). So we can work with $\tilde{\Theta}_k$ instead $\Theta_k$.
	
Now, we observe that the collection $\tilde{\Theta}_k$ is a (complete) metric subspace of the complete metric space $C(B_k,X)$ endowed by $d(f,g):=\displaystyle{\sup_{u\in B_k}}\|f(u)-g(u)\|$. Indeed, suppose $\gamma_n \rightarrow \gamma$ in $C(B_k,X)$ with $\gamma_n\in\tilde{\Theta}_k$. The semicontinuity of $I$ yields $$I(\gamma(u))\leq\liminf{I(\gamma_n(u))}\leq c_k-\frac{\varepsilon}{2},\,\,\,u \in \partial B_k.$$ Moreover, the action properties give
	\begin{equation*}
	\gamma(gu)=\lim\gamma_n(gu)=g\lim\gamma_n(u)\,\,\,\, \forall u \in B_k,\,\,\forall g \in G,
	\end{equation*}
	so that $\gamma \in \Gamma_G(B_k)$. Moreover, thanks to the continuity of $\Phi$, it is possible to find a sequence of positive numbers $\tau_n=o_n(1)$ such that
	\begin{equation}\label{Phi}
	\Phi(t\gamma_n(u)+(1-t)\gamma(u))\leq t\Phi(\gamma_n(u))+(1-t)\Phi(\gamma(u))+\tau_n\,\,\,\, \forall u \in\partial B_k,\,\, \forall t \in [0,1].
	\end{equation}
	More precisely $\tau_n:=2\max\{\tau^1_n,\tau^2_n\}$ with $$\tau^1_n:=\displaystyle{\sup_{u\in B_k,\,t \in [0,1]}}|\Phi(t\gamma_n(u)+(1-t)\gamma(u))-\Phi(\gamma(u))|$$ and $$\tau^2_n:=\displaystyle{\sup_{u\in B_k}} |\Phi(\gamma_n(u))-\Phi(\gamma(u))|.$$ The sentence in (\ref{Phi}) associated to the convexity of $\Psi$ implies
	\begin{equation}\begin{aligned}
	I(t\gamma_n(u)+(1-t)\gamma(u))&\leq tI(\gamma_n(u))+(1-t)I(\gamma(u))+\tau_n\\
	&\leq c_k-\frac{\varepsilon}{2}+\tau_n \leq c_k-\frac{\varepsilon}{4}\,\,\,\,\,\, \forall u \in \partial B_k,\,\,\forall t \in [0,1],
	\end{aligned}	
	\end{equation}
for $n$ sufficiently large.\par
	
	Thus ${\gamma_n}|_{\partial B_k} \approx \gamma|_{\partial B_k}$ via the equivariant homotopy $F(t,\cdot):=t\gamma_n(\cdot)+(1-t)\gamma(\cdot)$. Consequently  $\gamma|_{\partial B_k}\approx Id|_{\partial B_k}$, so that $\tilde{\Theta}_k$ is a complete metric subspace of $C(B_k,X)$ as claimed.
	The conclusion follows as in  \cite[Theorem 3.2]{Szulkin}. Using \cite[Lemma 3.1]{Szulkin}  and the definition of $c_k$, as $I$ is l.s.c., we have that the functional $\varphi:\tilde{\Theta}_k \to (-\infty,+\infty]$ given by 
	$$\varphi(\gamma):=\displaystyle{\sup_{u\in B_k}I(\gamma(u))}$$ is also a l.s.c. functional and bounded from below. Since $\tilde{\Theta}_k$ is a complete metric space, we can apply the classical Ekeland's Variational Principle recalled in Theorem \ref{PVE}, to the functional $\varphi$ with $\delta=\varepsilon$ and $\tau=1$. Then, we may take $\gamma \in \tilde{\Theta}_k$ such that $\varphi(\gamma) \leq c_k+\varepsilon$,  and
	\begin{equation}\label{PV1}
	\varphi(\eta)-\varphi(\gamma)\geq-\varepsilon d(\eta,\gamma)\,\,\,\,\forall \eta \in \tilde{\Theta}_k.
	\end{equation}
	It follows that $A:=\gamma(B_k)$ is a compact equivariant set with $$\displaystyle{\sup_{v\in A}I(v)}=\displaystyle{\sup_{u\in B_k}I(\gamma(u))}\leq c_k+\varepsilon,$$ so that $A$ verifies all the assumptions of the equivariant deformation lemma given in Lemma \ref{EDL}. Hence, let $\eta:=\alpha_s\circ\gamma,$ where $\alpha_s$ is the equivariant deformation given in Lemma \ref{EDL} and let us prove that $\eta \in \tilde{\Theta}_k$ for $s\approx 0^+$. Indeed $\eta \in \Gamma_G(B_k)$ and if $u\in \partial B_k$, by $iii)-iv)$ in Lemma \ref{DL} it follows that
	\begin{equation}\label{PV2}
	\left\{\begin{aligned}&I(\eta(u))=I(\alpha_s(\gamma(u)))\leq I(\gamma(u))\leq c_k-\frac{\varepsilon}{2},\quad I(\gamma(u))\in\left(c_k-\varepsilon, c_k-\frac{\varepsilon}{2}\right]
	\\
	&I(\eta(u))\leq I(\gamma(u))+2s\leq c_k-\frac{\varepsilon}{2},\quad I(u)\leq c_k-\varepsilon,
	\end{aligned} \right.
	\end{equation}
	so that 
	$$(I\circ\eta)|_{\partial B_k}\leq c_k-\frac{\varepsilon}{2}.$$ 
	Now, since $\alpha_s\circ\gamma$ can be viewed as an equivariant homotopy such that $(\alpha_s\circ\gamma)|_{\partial B_k}\approx\gamma|_{\partial B_k}$ in $I^{c_k-\frac{\varepsilon}{2}}$, it follows that
$$\eta|_{\partial B_k}\approx(\alpha_s\circ\gamma)|_{\partial B_k}\approx Id|_{\partial B_k}\quad \mbox{in}\quad I^{c_k-\frac{\varepsilon}{4}},$$ taking into account that $\gamma|_{\partial B_k}\approx Id|_{\partial B_k}$.\par
	\noindent Finally, since $\eta\in\tilde{\Theta}_k$, one has
the contradiction 
	\begin{equation}
	\begin{aligned}
	-\varepsilon s &\leq \varphi(\eta)-\varphi(\gamma)\\
	& =\displaystyle{\sup_{u\in B_k}}I(\alpha_s(\gamma(u)))-\displaystyle{\sup_{u\in B_k}}I(\gamma(u))\leq-2\varepsilon s,
	\end{aligned}
	\end{equation}
by using $i)$ and $v)$ of Lemma \ref{DL} and (\ref{PV1}). This contradiction guarantees that, for some $k_0$, $K_{c_k}\neq \emptyset$ for $k\geq k_0$ and thus the proof is complete since $c_k\geq b_k$.
\end{proof}	
	
\subsection{Minimax results involving the G-index Theory}
Let us recall the notation
$$N_\delta (A):=\{x\in X: d(x,A)\leq \delta\}.$$
The next technical lemma will be useful in the sequel.
\begin{lemma}\label{TL}
	Let $I:=\Phi+\Psi \in (H_0)$ be a functional which satisfies the compactness $(\rm PS)$ condition. Moreover, let $(c_j)$ be a real sequence such that $c_j \rightarrow c \in \mathbb{R}$. Then, given $\delta>0,$ there exists $j_0 \in \mathbb{N}$ such that
	$$ K_{c_j}\subset N_\delta(K_c),$$
for every $j\geq j_0.$
\end{lemma}

\begin{proof}
Arguing by contradiction, assume that there exist a subsequence $(c_{j_k})$ of $(c_j)$, a number $\delta_0>0$, and a sequence $(u_k)$ with $u_k \in K_{c_{j_k}}$ such that
	\begin{equation}\label{DI}
	d(u_k,K_c)>\delta_0,\,\,\, \forall k \in \mathbb{N}.
	\end{equation}
	The definition of $K_{c_{j_k}}$ immediately yields
	\begin{equation}\label{CPk}
	\langle \Phi'(u_k),v-u_k\rangle+\Psi(v)-\Psi(u_k) \geq 0\,\,\, \forall v \in X,
	\end{equation}
	as well as
	$$
	I(u_k)=c_{j_k}\rightarrow c,
	$$
	so that $(u_k)$ is a $(\rm PS)_c$ sequence for the functional $I$. Now, the $(\rm PS)$ condition ensures the existence of $u_0 \in X$ and a subsequence of $(u_k)$, still denoted again by $(u_k)$, such that
	$$
	u_k \rightarrow u_0 \quad \mbox{in} \quad X.
	$$
	Now, taking $v=u_0$ in (\ref{CPk}), we get $\limsup \Psi(u_k) \leq \Psi(u_0)$. This together with the semicontinuity property of $\Psi$ gives $\lim \Psi(u_k)=\Psi(u_0)$, and so, $u_0 \in K_c.$ Consequently $d(u_k,K_c) \to 0$ as $k \to \infty$, against (\ref{DI}).
\end{proof}	

The next result extends a result due to Heinz \cite[Proposition 2.2]{Heinz} for functionals $I\in (H_0)$.

\begin{theorem}\label{HT} Suppose that $I:=\Phi+\Psi\in(H_0)$ satisfies the $(\rm PS)$ condition, $\Phi$ and $\Psi$ are $G$-invariant functionals, $I(0)=0$, $(G_0)$ holds and that the following property in the $G$-index also holds:
$$
\gamma(A)<\infty \quad \mbox{for all compact set} \,\, A \in \Sigma.  \leqno{(G_*)}
$$
Moreover, fix
	$$
	 c_j:=\inf_{A \in \Gamma_j}\sup_{u \in A}I(u),
	 $$
and assume the following conditions:
\begin{itemize}
	\item[$i)$] $-\infty<c_j$ for every $j\in\mathbb{N};$
	\item[$ii)$] Given $j \in \mathbb{N},$ there exists $A \in \Sigma$ such that
$$\gamma_G(A)\geq j\quad \mbox{and}\quad \displaystyle{\sup_{u \in A}}I(u)<0,$$ where $A\neq\emptyset$ is a compact set.
\end{itemize}
Then, the numbers $c_j$ are negative critical values of $I$ with $\lim c_j =0.$	
\end{theorem}
\begin{proof}
	We first notice that conditions $i)$ and $ii)$ imply that $-\infty<c_j<0$. Now, we also point out that, by reading carefully the argument in \cite[Theorem 4.3]{Szulkin}, we can show that the sequence $(c_j)$ consists of critical values of $I$. In fact, the proof of \cite[Theorem 4.3]{Szulkin} only depends on the properties $i)-iv)$ in Proposition \ref{GP} in the case that $G=\mathbb{Z}_2$ where  $\gamma_G$ coincides with the genus of a symmetric set mentioned in Remark \ref{GT}. In this way, in view of the Proposition \ref{GP}, we can apply the argument used in \cite[Theorem 4.3]{Szulkin} in our case. It remains to show that $c_j \rightarrow 0$. Now, the definition of $c_j$ yields
	$$
	c_j\leq c_{j+1},\,\,\, \forall j \in \mathbb{N}.
	$$
	Hence, if $c_j \nrightarrow 0$ for $j \to \infty$, there exists $c<0$ such that $c_j\rightarrow c$. The $(\rm PS)$ condition ensures that $K_c$ is compact. Moreover, the assumptions on $I$ ensure that $K_c$ is symmetric and  $0\notin K_c$. Thereby, $K_c \in \Sigma$ and, by following the idea of Lemma \ref{TL}, as $c_j\rightarrow c$ and $K_{c_j}\neq\emptyset$, one has that $K_c\neq \emptyset$. By $vi)$ of Proposition \ref{GP} there is $\delta>0$ such that $\gamma_G(N_{2\delta}(K_c))=\gamma_G(K_c)$; note that $N_{\delta}(K_c)\neq\emptyset$. By $(G_*)$, we can assume that $\gamma_G(K_c)= p$ 
	$$ \begin{aligned}
	\varphi_j:\,&\Gamma_j\,\,\,\,\rightarrow\,\,\, (-\infty,\infty]\\
	&A \quad\longmapsto \varphi_j(A):=\sup_{u \in A}I(u).
	\end{aligned}
	$$
	Clearly $\varphi_j$ is $l.s.c$ since $I$ is too. Set
	$$\varepsilon_0:=\min\{1,\delta,-c\}$$
	and take $\varepsilon\in(0,\varepsilon_0)$ as in Lemma \ref{DL}. Now, let $A_1 \in \Gamma_{j+p}$ be such that
	\begin{equation*}
	c_{j+p}\leq \varphi_{j+p}(A_1)<c_{j+p}+\frac{\varepsilon^2}{2}.
	\end{equation*}
	Since $c_j\rightarrow c$, it follows that, for a convenient $j_0 \in \mathbb{N}$,
	$$ \varphi_{j+p}(A_1)<c_{j+p}+\frac{\varepsilon^2}{2}\leq c+\frac{\varepsilon^2}{2}\leq c_j+\varepsilon^2<c_j+\varepsilon<0,
	$$
	for $j\geq j_0$.
	Hence, by fixing $j=j_0$, we get $0\notin A_1$ and $\gamma_G(A_1)\geq j_0+p$ by Proposition \ref{GP1}. If we set $A_2:=\overline{A_1\setminus N_{2\delta}(K_c)}$ we also have
	$$
	\sup_{u \in A_2}I(u)\leq\sup_{u \in A_1}I(u)<c_{j_0}+\varepsilon^2<0,
	$$
	so that $0\notin A_2$ and $\gamma_G(A_2)\geq (j_0+p)-p=j_0$ by Proposition \ref{GP} - Part $iv)$. Consequently $A_2 \in \Gamma_{j_0}$. Employing Theorem \ref{PVE} to the function $\varphi_{j_0}:\Gamma_{j_0} \rightarrow (-\infty,\infty]$ (note that $\Gamma_{j_0}$ is complete), there exists $A \in \Gamma_{j_0}$ such that
	$$c_{j_0}\leq \sup_{u \in A} I(u)=\varphi_{j_0}(A)\leq\varphi_{j_0}(A_2)<c_{j_0}+\varepsilon,\,\,\, d_H(A,A_2) \leq \varepsilon$$
	as well as
	\begin{equation}\label{IE}
	\varphi_{j_0}(B)-\varphi_{j_0}(A) \geq -\varepsilon d_H(A,B) \,\,\, \forall B \in \Gamma_{j_0}.
	\end{equation}
	
	\noindent Since Lemma \ref{TL} gives $K_{c_{j_0}} \subset N_\delta(K_c)$ for $j_0\approx\infty$, by setting $N=N_\delta(K_c)$ we derive $A\cap N=\emptyset$, taking into account that $\varepsilon<\delta$. These informations ensure that $A$, $N$ and $K_{c_{j_0}}$ verify the hypothesis of deformation lemma (Lemma \ref{DL}).
	
	Thus by Lemma \ref{EDL} the existence of an equivariant deformation $\alpha_s$ is obtained. In this way, if we set $B:=\alpha_s(A)$, on account of Proposition \ref{GP} - Part $i)$, one has $B \in \Gamma_{j_0}.$ Now, combining the properties of $\alpha_s$ with (\ref{IE}) we derive the contradiction
	$$-2\varepsilon s\geq \varphi(B)-\varphi(A)\geq-\varepsilon s.$$
 This completes the proof.	
\end{proof}

\begin{remark} Here, we would like to point out that the condition $(G_*)$ that was assumed in the last theorem is not necessary when the group is finite, see Proposition \ref{GP}-$v)$.
\end{remark}

In the sequel, we state the last abstract result this section that complements the study made by Szulkin in \cite[Corollary 4.8]{Szulkin}. 

\begin{theorem}\label{SP}
Suppose that $I:=\Phi+\Psi\in(H_0)$ satisfies the $(\rm PS)$ condition, $\Phi$ and $\Psi$ are $G$-invariant functionals, $I(0)=0$ and that $(G_0)-(G_*)$ hold. Assume that there exist subspaces $Y,Z$ of $X$ such that $X=Y\oplus Z,$ $\dim Y<\infty,$ $Z$ is closed and
\begin{itemize}
	\item[$i)$] There are numbers $r,\rho>0$ such that $ I|_{\partial B_r(0)\cap Z} \geq \rho;$
	\item[$ii)$] For each positive integer $k$ there is a $k$-dimensional subspace $X_k$ of $X$ such that $I(u) \rightarrow -\infty$ as $\|u\|\rightarrow \infty$ with $u \in X_k$.
\end{itemize}	
	Then $I$ has infinitely many critical values. Furthermore, if there is a $c_0>0$ such that $I^{-c_0}$ has no critical points, then there exists a sequence $(c_j)$ of critical values of $I$ with $c_j\rightarrow \infty$.
\end{theorem}

In order to prove Theorem \ref{SP} some notations are introduced. To this aim, let us fix $c_0>0$ such that $I^{-c_0}$ has no critical points and set $M_k:=\overline{B}_{R_k}(0)\cap X_k$ with $R_k>r$ and $I|_{\partial M}\leq -c_0$. Now, let us define the following sets
$$\mathcal{F}:=\{\eta\in C(M_k,X):\eta\, \text{is odd},\, \eta|_{\partial M_k}\approx Id|_{\partial M_k}\,\text{by an odd homotopy}\},
$$
for each $j \in \mathbb{N}$ and $k\geq j$,
$$\tilde{\Lambda}^k_j:=\left\{\begin{aligned}\eta(M_k\setminus U):\eta\in\mathcal{F}&,\, U\,\text{is symmetric and open in}\, M_k,\,U\cap\partial M_k=\emptyset,\\ 
& \text{with}\, \gamma_G(W)\leq k-j,\, \text{for}\,W\in\Sigma,\,W\subset U.\end{aligned}\right\}$$
and
$$\tilde{\Lambda}_j:=\displaystyle{\bigcup_{k\geq j}}\, \tilde{\Lambda}^k_j.$$
Finally, for each $j \in \N$, we fix 
$$\Lambda_j:=\{A\subset X: A\,\text{is compact, symmetric and for each open}\, U\supset A,\,\text{there is}\, A_0\in\tilde{\Lambda}_j,\,A_0\subset U\}.
$$
and 
$$c_j:=\inf_{A \in \Lambda_j}\sup_{u \in A}I(u).$$

By applying the same arguments used in \cite[Theorem 4.4, Lemma 4.6]{Szulkin} we can prove that $\Lambda_j$ verifies the properties $i)-v)$ below (it suffices to replace $\gamma$ in that reference by the $G$-index $\gamma_G$ and use the Proposition \ref{GP}).

\begin{lemma}\label{LC}
	The sets $\Lambda_j$ defined above satisfy the following claims:
\begin{itemize}
	\item[$i)$] $(\Lambda_j,d_H)$ is a complete metric space$;$
	\item[$ii)$] $c_j \geq \rho$, for all $j > \text{\rm dim} Y$;
	\item[$iii)$] $\Lambda_{j+1}\subset \Lambda_j$$;$
	\item[$iv)$] Let $A \in \Lambda_j$ and $W$ be a closed symmetric set containing $A$ in its interior. Moreover, if $\alpha:W\rightarrow X$ is an odd mapping such that $\alpha|_{W\cap I^{-c_0}}\approx Id|_{W\cap I^{-c_0}}$ by an odd homotopy, then $\alpha(A) \in \Lambda_j;$
	\item[$v)$] For each compact $B$ with $B \in \Sigma,$ $\gamma_G(B)\leq p,$ $I|_B>-c_0,$ there exists a number $\delta_0>0$ such that $A\setminus\text{int}(N_\delta(B))\in \Lambda_j,$ for $A \in \Lambda_{j+p},$ $\delta\in(0,\delta_0)$.
\end{itemize}
\end{lemma}

\noindent We point out that step $v)$ is different with respect to the statement of \cite[Lemma 4.6]{Szulkin}. However, the main assertion is a direct consequence of the arguments proved there.

\begin{proof}[Proof of Theorem \ref{SP}]
	The first part of the proof can be derived by \cite[Corollary 4.8]{Szulkin}. Hence, it remains to show that $c_j \rightarrow \infty$ as $j \to \infty$. Now, by Lemma \ref{LC} - Part $iii)$, it follows that
	$$ c_j\leq c_{j+1}\,\,\,\, \forall j \in\mathbb{N}.$$
	Thus, if $c_j  \nrightarrow \infty$, by $ii)$ of last lemma, there exists $c>0$ such that $c_j\rightarrow c$. Arguing as in the proof of Theorem \ref{HT}, we deduce that $K_c$ is a compact symmetric set with $0 \notin K_c$ and $K_c\neq \emptyset$. Hence, for a convenient $\delta>0$, one has $\gamma_G(N_{2\delta}(K_c))=\gamma_G(K_c)=:p\in\mathbb{N}.$ Now, set
	$\varepsilon_0:=\min\{1,\,\delta\},$
	take $\varepsilon \in (0,\varepsilon_0)$ as in Lemma \ref{DL} and define
	
	$$\begin{aligned}
	\varphi_j:\,&\Lambda_j\,\,\,\rightarrow\,\,\,\,\,\,\,(-\infty,\infty]\\
	& A\quad\longmapsto \varphi(A):=\sup_{u \in A}I(u).
	\end{aligned}
	$$
	Clearly $\varphi_j$ is a l.s.c. functional that is bounded from below for every $j\in\mathbb{N}$. Hence, let $A_1 \in \Lambda_{j+p}$ be such that
	$$\varphi_{j+p}(A_1)<c_{j+p}+\frac{\varepsilon^2}{2}.$$
	Consequently, for some $j_0 \in \mathbb{N}$, 
	$$\varphi_{j+p}(A_1)<c_j+\varepsilon,$$
for $j\geq j_0$.
	Now, if $A_2:=\overline{A_1\setminus\text{int}(N_{2\delta}(K_c))}$, by Part $v)$ of Lemma \ref{LC} we have $A_2 \in \Lambda_{j_0}$ and $\varphi_{j_0}(A_2)\leq\varphi_{j_0}(A_1)$. Moreover, by Theorem \ref{PVE}, there exists $A \in \Lambda_{j_0}$ such that
	$$\varphi_{j_0}(A)\leq\varphi_{j_0}(A_2)<c_{j_0}+\varepsilon\,\,\,\,\,d_H(A,A_2)\leq\varepsilon$$
	as well as
	\begin{equation}\label{EI}
	\varphi_{j_0}(B)-\varphi_{j_0}(A)\geq -\varepsilon d_H(B,A)\,\,\,\,\,\forall B\in\Lambda_{j_0}.
	\end{equation}
	
	\noindent If we set $N:=N_\delta(K_c)$, Lemma \ref{TL} implies that $K_{c_{j_0}}\subset N$ if $j_0\approx\infty$. The definition of $\varepsilon_0$ yields $A\cap N=\emptyset$ and
	$$c_{j_0}\leq \sup_{u \in A}I(u) <c_{j_0}+\varepsilon.$$
	\noindent Then, we can apply Lemma \ref{EDL} to obtain an equivariant deformation $\alpha_s$. If we set $B:= \alpha_s(A)$, by Part $vi)$ of Lemma \ref{DL} and Part $iv)$ of Lemma \ref{LC}, one has $B \in \Lambda_{j_0}$. Finally, a contradiction is achieved by replacing $B$ in (\ref{EI}) and arguing as in the proof of Theorem \ref{HT}.	
\end{proof}

\section{Some Applications to elliptic problems}

In this section we illustrate how the abstract results of the previous section can be applied  to establish the existence of infinitely many solutions for some classes of elliptic problems.

\subsection{A logarithmic variational inclusion problem}
In this subsection, we study the existence of infinitely many solutions for the logarithmic inclusion problem
$$
\left\{\begin{aligned}
-&\Delta u + u + \partial F(x,u) \ni u\log u^2,  \;\;\mbox{in}\;\;\mathbb{R}^{N} \\
& u \in H^{1}(\mathbb{R}^{N}),
\end{aligned}
\right. \leqno{(P_1)}
$$
where $F(x,t):= \displaystyle{\int_{0}^{t}}f(x,s)ds$ is a convex locally Lipschitz function with $F(x,\cdot)\geq 0$ for all $x \in \mathbb{R}^N$ and satisfying some technical conditions. 

In the sequel, we assume that $f$ is a $N$-measurable function that satisfies the following conditions:
\begin{itemize}
\item[$(f_1)$] There is a nonnegative and radial function $h \in L^1(\mathbb{R}^N)\cap L^\infty(\mathbb{R}^N)$ satisfying 
$$|f(x,t)|\leq h(x)|t|,\,\,\,\, \forall x \in \mathbb{R}^N \quad \mbox{and} \quad \forall t \in \R.$$
\item[$(f_2)$] $f(x,-t)=-f(x,t)$ and $f(|x|,t)=f(x,t)$ for all $x \in \mathbb{R}^N$ and $t\in\mathbb{R}$.
\item [$(f_3)$] There is $C>0$ such that for any $\eta_t \in \partial F(x,t)$ it holds
$$ F(x,u)-\frac{1}{2}\eta_tt\geq-Ch(x),\,\,\,\text{a.e}\,\,\,x \in \mathbb{R}^N,\,\forall\,\,t \in \mathbb{R}.$$
\end{itemize}
\begin{example}[A function satisfying $(f_1)-(f_3)$]
	Consider $F(x,t):= h(x)\displaystyle{\int_{0}^{t}}H(|s|-a)s\,ds$, where $a>0$, $h \in L^1(\mathbb{R}^N)\cap L^\infty(\mathbb{R}^N)$ is nonnegative and radial and $H$ is the Heaviside function, i.e.,
	$$
	H(t):=\left\{\begin{aligned}
	&0,\,\,\,\, t \leq 0\\
	&1,\,\,\,\, t>0.
	\end{aligned}
	\right.
	$$
	Notice that, in this case,
	$$
\partial F(x,t)=h(x)\left\{\begin{aligned}
	&\{s\}\,\,\,\, &|s|>a,\\
	&[-a,0]\,\,\,& s=-a,\\
	&[0,a]\,\,\,&s=a,\\
	&\{0\}\,\,\,&|s|<a.
	\end{aligned}
	\right.
	$$
	The reader is invited to note that $(f_1)-(f_3)$ occur for this example.  
\end{example}

Now, consider the energy functional associated to  problem $(P_1)$ given by
\begin{equation*}
I(u):=\frac{1}{2}\int_{\mathbb{R}^N}(|\nabla u|^2+|u|^2) + \int _{\mathbb{R}^N}F(x,u) - \int_{\mathbb{R}^N}L(u),\,\,\,  u \in H^1(\mathbb{R}^N),
\end{equation*}
where $$L(t):= -\displaystyle\frac{t^2}{2}+\frac{t^2\log t^2}{2}\quad \forall\,t\in \R$$. 

Hereafter, we make use of the approach given in \cite{Alves-de Morais, Alves-Ji, Ji-Szulkin} to decompose $I$ as a sum of a $C^1$ functional and a convex l.s.c. functional. To this aim, fixed $\delta>0$ sufficiently small, we set
$$
F_1(s):=\left\{\begin{aligned}
&0  \quad \,& s=0\\
-\frac{1}{2} &s^2 \log s^2\quad &0<|s|<\delta\\
-\frac{1}{2} &s^2 (\log \delta^2 +3) + 2\delta|s| - \frac{\delta^2}{2}&|s|\geq \delta
\end{aligned}
\right.
$$
and
$$
F_2(s):=	\left\{\begin{aligned}
&0  \quad \,& s=0\\
-\frac{1}{2} &s^2 \log \left(\frac{s^2}{\delta^2}\right) + 2\delta|s| -\frac{3}{2}s^2-\frac{\delta^2}{2}\,&|s|\geq \delta
\end{aligned}
\right.
$$	
for every $s\in \R.$
Therefore
$$F_2(s)-F_1(s)=\frac{1}{2}s^2 \log s^2\,\,\,\, \forall s \in \mathbb{R},$$
and
\begin{equation}\label{Functional2}
I(u)=\frac{1}{2}\|u\|^2+\int_{\mathbb{R}^N}F(x,u)+\int_{\mathbb{R}^N}F_1(u)-\int_{\mathbb{R}^N}F_2(u)\,\,\,\, u \in H^1(\mathbb{R}^N),
\end{equation}
where $\|\cdot \|$ denotes the norm in $H^1(\mathbb{R}^N)$ induced by the inner product given by
\begin{equation*}
\langle u, v \rangle:=\int_{\mathbb{R}^N} (\nabla u \nabla v+2uv),\,\,\, u,v \in H^1(\mathbb{R}^N).
\end{equation*}	
According to \cite[Section 2]{Alves-de Morais} and \cite[Section 2]{Ji-Szulkin} the functions $F_1$ and $F_2$ satisfy the following conditions:
\begin{itemize}
\item[($A_1)$] $F_1$ is an even function with $F_1'(s)s\geq 0$ and $F_1\geq 0$. Moreover $F_1 \in C^1(\mathbb{R},\mathbb{R})$ and it is also convex if $\delta \approx 0^+$;
\item[($A_2)$] $F_2 \in C^1(\mathbb{R},\mathbb{R})$ and for each $p \in (2,2^*)$, there exists $C=C_p>0$ such that
$$ |F_2'(s)|\leq C|s|^{p-1}\,\,\,\, \forall s \in \mathbb{R}.$$
\end{itemize}
 Now, by $(A_1)$ and $(A_2)$, it is easily seen that $I \in (H_0)$ with
$$\Phi(u):= \frac{1}{2}\|u\|^2 - \int_{\mathbb{R}^N} F_2(u)$$
and
$$\Psi(u):=\int_{\mathbb{R}^N}F_1(u)+\int_{\mathbb{R}^N}F(x,u).$$

\noindent We notice that $\Psi=\Psi_1+\Psi_2$, where
$$\Psi_1(u):=\displaystyle{\int _{\mathbb{R}^N}}F_1(u)\,\,\,
\text{and}\,\,\, \Psi_2(u):=\displaystyle{\int_{\mathbb{R}^N}}F(x,u).$$

\indent  
Direct arguments and \cite[Lemma 2.1]{Alves-de Morais} ensure the validity of the next result.
\begin{lemma}\label{Psi1}
Let $\Psi_1:H^1(\mathbb{R}^N)\rightarrow(-\infty,\infty]$ be the functional defined above. Then
\begin{itemize}
	\item[$i)$] $D(I)=D(\Psi_1)$, that is $I(u)<\infty$ if and only if $\Psi_1(u)<\infty$.
	\item[$ii)$] Let $\Omega \subset \mathbb{R}^N$ be a bounded domain with regular boundary. Then the functional
	\begin{equation}
	\tilde{\Psi}_1(u)=\int_{\Omega}F_1(u)
	\end{equation}
	belongs to $C^1(H^1(\Omega),\mathbb{R})$.
\end{itemize}
\end{lemma}

Moreover, according to \cite{Chang1}, the structural conditions on $F$ assure that the functional $\Psi_2:H^1(\mathbb{R}^N)\rightarrow\mathbb{R}$
is convex and l.s.c. as well as $\Psi_2 \in Lip_{loc}(H^1(\mathbb{R}^N),\mathbb{R})$.\par
\smallskip
From now on, for each $u \in H^1(\mathbb{R}^N)$, let us consider the functional  $\varphi_1^u$ defined by
\begin{equation}\label{Fam}
\langle \varphi_1^u,v\rangle:=\int_{\mathbb{R}^N}F_1'(u)v,\,\,\,\, \forall v \in C_0^\infty(\mathbb{R}^N).
\end{equation}

If
$$\|\varphi_1^u\|:=\displaystyle{\sup_{v \in C_0^\infty(\mathbb{R}^N),\, \|v\|\leq 1}} \langle \varphi_1^u,v \rangle <\infty,$$
then $\varphi_1^u$ can be extended to a continuous linear functional on $H^1(\mathbb{R}^N)$.\par
 Moreover, if  $\tilde{I}:H^{1}(\mathbb{R}^N) \to (-\infty,+\infty]$ denotes the functional given by
$$
\tilde{I}(u):= \frac{1}{2}\|u\|^2+\int_{\mathbb{R}^N}F_1(u)-\int_{\mathbb{R}^N}F_2(u),
$$
then $\tilde{I} \in (H_0)$ and $I=\tilde{I}+\Psi_2$.\par
By \cite[Lemma 2.2 and Corollary 2.1]{Alves-de Morais} the following lemma holds.
\begin{lemma} \label{Fam2}
	If $u \in D(\tilde{I})$ and $\|\varphi_1^u\|<\infty$ then there is a unique functional in $\partial \tilde{I}(u)$,  denoted by $\tilde{I}'(u),$ such that
	\begin{equation}\label{Fam2.1}
	\tilde{I}'(u)(v)= \langle \Phi'(u),v \rangle +\int_{\mathbb{R}^N}F'_1(u)v\,\,\,\, \forall v \in C_0^\infty(\mathbb{R}^N).
	\end{equation}
	Furthermore, $F_1'(u)u\in L^1(\mathbb{R}^N),$ and
	\begin{equation}
	\tilde{I}'(u)(u) = \int_{\mathbb{R}^N}(|\nabla u|^2+|u|^2)-\int_{\mathbb{R}^N}u^2 \log u^2,
	\end{equation}
	as well as
	\begin{equation}
	\tilde{I}(u)-\frac{1}{2}\tilde{I}'(u)(u)=\frac{1}{2}\int_{\mathbb{R}^N}|u|^2.
	\end{equation}
\end{lemma}

\begin{remark}\rm{
Lemma \ref{Fam2} remains valid if we take $\tilde{J}:=\tilde{I}|_{H_{rad}^1(\mathbb{R}^N)}$. Indeed, the arguments used in \cite[Lemma 2.2 and of Corollary 2.1]{Alves-de Morais} can be adapted to the radial space $H_{rad}^1(\mathbb{R}^N)$ by taking $\{\varphi_1^u\}\subset H_{rad}^1(\mathbb{R}^N)$ and
	$$
	\langle \varphi_1^u,v \rangle=\int_{\mathbb{R}^N}F_1'(u)v\,\,\,\,\, v \in C_{0,\,rad}^\infty(\mathbb{R}^N).
	$$}
\end{remark}
\noindent The notion of solution for problem $(P_1)$ requires some comments. To this aim, let us define the functions
\begin{equation}\label{Under}
\underline{f}(x,t):= \lim_{r\downarrow 0} ess \inf\{f(x,s): |s-t|<r\}
\end{equation}
and
\begin{equation}\label{Over}
\overline{f}(x,t):= \lim_{r\downarrow 0} ess \sup\{f(x,s): |s-t|<r\}.
\end{equation}

According to \cite[Section 2]{Chang1} if $F(x,t)=\displaystyle{\int_{0}^{t}}f(x,s)\,ds$, then 
$$\partial F(x,t)=[\underline{f}(x,t),\overline{f}(x,t)],$$
and the following definition makes sense. 
\begin{definition}\label{AWS}
A function
	$u \in H^1(\mathbb{R}^N)$ is said to be a solution of $(P_1)$ if $u^2\log u^2 \in L^1(\mathbb{R}^N)$ and there exists $\rho \in L^2(\mathbb{R}^N)$ such that 
$$\rho(x) \in [\underline{f}(x,u(x)),\overline{f}(x,u(x))]\quad \mbox{ a.e in}\quad \mathbb{R}^N$$ and
	\begin{equation}\label{Solution}
	\int_{\mathbb{R}^N} (\nabla u \nabla \phi + u\phi)+\int_{\mathbb{R}^N}\rho\phi=\int_{\mathbb{R}^N}u\log u^2\phi, \,\,\, \forall \phi \in C_0^{\infty}(\mathbb{R}^N).
	\end{equation}
\end{definition}

Next technical result is given in \cite[Lemma 4.1]{AlvesJVJA2}.

\begin{lemma}\label{S3}
The functions  $\underline{f}$ and $\overline{f}$ are $N$-measurable functions, ${\Psi}_2 \in Lip_{loc}(L^2(\mathbb{R}^N),\mathbb{R})$
	and 
	\begin{equation}\label{Inclusion}
	\partial {\Psi}_2(u)\subseteq\partial F(x,u)=[\underline{f}(x,u(x)),\overline{f}(x,u(x))],
	\end{equation}
for every $u \in L^2(\mathbb{R}^N)$.
\end{lemma}

The inclusion in (\ref{Inclusion}) has the following meaning: for each $\eta\in \partial \Psi_2(u)$ there is a function $\tilde{\eta} \in L^2(\mathbb{R}^N)$ such that 
\begin{itemize}
\item[$i)$] $\eta(v)=\displaystyle{\int_{\mathbb{R}^N}}\tilde{\eta}v\,\,\,\, \forall v \in L^2(\mathbb{R}^N);$
\item[$ii)$] $\tilde{\eta}(x) \in [\underline{f}(x,u(x)),\overline{f}(x,u(x))]$ a.e. in $\mathbb{R}^N$.
\end{itemize}

\noindent Our next step is proving that the critical points of $I$ in the sense given in Definition \ref{cp} are solutions of $(P_1)$.

\begin{lemma}\label{CritSolution}
	Every critical point of the functional $I$ is a solution of $(P_1)$.
\end{lemma}
\begin{proof}
	Suppose that $u \in D(I)$ is a critical point of $I$, that is,
	\begin{equation}\label{CPI}
\begin{split}
\int_{\mathbb{R}^N} (\nabla u \nabla(v-u)+2u(v-u)) & +\int_{\mathbb{R}^N}(F(x,v)-F(x,u)) \\
& \geq \int_{\mathbb{R}^N}F'_2(u)(v-u) - \int_{\mathbb{R}^N}(F_1(v)-F_1(u)),
\end{split}
\end{equation}
for every $v \in H^1(\mathbb{R}^N).$
	\noindent The last sentence means that the functional $-\Phi'(u)$ belongs to $\partial \Psi(u)$. Hence, by choosing $v=u+t\phi$, $t>0$, $\phi \in C_0^{\infty}(\mathbb{R}^N)$, we find
	\begin{equation}
	\int_{\mathbb{R}^N}\frac{1}{t}(F(x,u+t\phi)-F(x,u))+\int_{\mathbb{R}^N}\frac{1}{t}(F_1(u+t\phi)-F_1(u))\geq \langle-\Phi'(u),\phi \rangle,
	\end{equation}
	which is equivalent to
	\begin{equation}
	\frac{1}{t}[\Psi_2(u+t\phi)-\Psi_2(u)]+\int_{\mathbb{R}^N}\frac{1}{t}(F_1(u+t\phi)-F_1(u))\geq \langle-\Phi'(u),\phi \rangle.
	\end{equation}
	As $\Psi_2$ is convex, when $t \rightarrow 0^+$, the Lemmas \ref{G} and \ref{Psi1} imply that
	\begin{equation}\label{In}
	\Psi_2^\circ(u,\phi)+\int_{\mathbb{R}^N}F_1'(u)\phi\geq \langle-\Phi'(u),\phi \rangle.
	\end{equation}
	
\noindent Replacing $\phi$ with $-\phi$ in (\ref{In}) and by using Lemma \ref{G} it follows that
	\begin{equation}
	\Psi_2^\circ(u,-\phi)-\langle\Phi'(u),\phi \rangle \geq \int_{\mathbb{R}^N}F_1'(u)\phi.
	\end{equation}
	Then, according to the notation introduced in (\ref{Fam}), one has
	\begin{equation}\label{Psi2}
	\Psi_2^\circ(u,-\phi)-\langle\Phi'(u),\phi \rangle\geq\langle \varphi_1^{u},  \phi \rangle  .
	\end{equation}
	
	\noindent Let us prove now the following claim that will be crucial in the rest of the proof.
	\begin{claim}\label{Sup}
		$\displaystyle{\sup_{\phi \in C_0^\infty(\mathbb{R}^N),\, \|\phi\|\leq 1}}\Psi_2^\circ(u,\phi)< \infty.$
	\end{claim}
	
	\noindent Indeed, by Lemma \ref{S3}, for each $\phi \in  C_0^\infty(\mathbb{R}^N)$ with $\|\phi\|\leq 1$, there is $\tilde{\eta}_\phi \in L^2(\mathbb{R}^N)$ such that  $\tilde{\eta}_\phi(x) \in [\underline{f}(x,u(x)),\overline{f}(x,u(x))]$ and
	$$
	\Psi_2^\circ(u,\phi)= \int_{\mathbb{R}^N}\tilde{\eta}_\phi\phi.
	$$
Now, by $(h_1)$ and $(f_1)$, there exists a constant $C:=C(u,h)>0$, independent of $\phi$, such that
	$$\begin{aligned}
	\left|\int_{\mathbb{R}^N}\tilde{\eta}_\phi\phi\right|\leq C\|\phi\|.
	\end{aligned}
	$$
The Claim \ref{Sup} immediately follows.
	\noindent Now, we observe that Claim \ref{Sup} together with inequality (\ref{Psi2}) give
		$$
	\sup_{\phi \in C_0^\infty(\mathbb{R}^N),\, \|\phi\|\leq 1}\langle\varphi_1^{u},\phi\rangle < \infty.
	$$
	Consequently, the Hahn-Banach's extension theorem ensures that the functional $\varphi_1$ admits an extension, still denoted by $\varphi_1$, to a continuous linear functional on $H^1(\mathbb{R}^N)$. Moreover, Lemma \ref{usc}, inequality (\ref{In}) and the density of $C_0^\infty(\mathbb{R}^N)$ in $H^1(\mathbb{R}^N)$ yield
	\begin{equation}
	\langle-\Phi'(u)-\varphi_1^{u}, v\rangle \leq \Psi_2^\circ(u,v)\,\,\,\, \forall v \in H^1(\mathbb{R}^N),
	\end{equation}
	that is,
	\begin{equation}
	-\Phi'(u) -\varphi_1^{u} \in \partial \Psi_2 (u).
	\end{equation}
	Thus, there exists $\varphi_2 \in \partial \Psi_2 (u)$ such that $-\Phi'(u) -\varphi_1^{u} = \varphi_2$. Now, by Lemma \ref{S3}, there exists $\rho \in L^2(\mathbb{R}^N)$  such that $\rho(x) \in [\underline{f}(x,u(x)),\overline{f}(x,u(x))]$ a.e. in $\mathbb{R}^N$ and
	$$
	\langle \varphi_2, v \rangle=\int_{\mathbb{R}^N}\rho v,\,\,\, \forall v \in H^1(\mathbb{R}^N).
	$$
	Hence	
	$$\langle -\Phi'(u),v\rangle = \langle \varphi_1^{u},v\rangle+\int_{\mathbb{R}^N}\rho v\,\,\, \forall v \in H^1(\mathbb{R}^N).$$
	
	\noindent Taking $v=\phi \in C_0^\infty(\mathbb{R}^N)$ in the equation above, one has
	\begin{equation}
	\int_{\mathbb{R}^N}\rho \phi+\int_{\mathbb{R}^N}F'_1(u) \phi=\langle-\Phi'(u),\phi\rangle\, \,\,\, \forall \phi \in C_0^\infty(\mathbb{R}^N),
	\end{equation}
which finishes the proof. 		
\end{proof}

On account of Lemma \ref{CritSolution} problem $(P_1)$ has infinitely many solutions, which is an immediate consequence of following result.

\begin{theorem}\label{Theorem}
	The functional $I$ has a sequence of critical points $(u_n)$ such that $I(u_n) \rightarrow \infty$ as $n\rightarrow \infty$. Hence, problem $(P_1)$ has infinitely many nontrivial solutions.
\end{theorem}

	The proof of Theorem \ref{Theorem} is divided into several lemmas. To this goal, let $O(N)$ be the orthogonal group in $\mathbb{R}^N$. So, by using a standard change of variable, it is easy to check that the functional $I$ is $O(N)$-invariant. Moreover, the space of invariant elements of $H^1(\mathbb{R}^N)$ under the natural action of $O(N)$ is the subspace $H_{rad}^1(\mathbb{R}^N)$ of radial functions of $ H^1(\mathbb{R}^N)$. The Palais' Principle given in Theorem \ref{PC} ensures that the critical points of $J:=I|_{H_{rad}^1(\mathbb{R}^N)}$ are also critical points of the functional $I$. We also notice that the functional $J$ is a $\mathbb{Z}_2$-invariant since $J$ is an even functional. 
Hence, Theorem \ref{Theorem} will be proved by using Theorem \ref{Fountain1} (with $G=\mathbb{Z}_2$) and the action described in Example \ref{ex1}. A key ingredient in the  proof is the Sobolev compact embedding
	\begin{equation}\label{Embedding}
	H_{rad}^1(\mathbb{R}^N)\hookrightarrow L^p(\mathbb{R}^N), \quad \forall p\in(2,2^*).
	\end{equation}
See \cite[Corollary 1.26]{Willem} for additional comments and remarks about the embedding above.\par
	
	Let us prove the following preliminary result.
	
	\begin{lemma}\label{Sub}
		Let $(u_n)$ be a $(\rm PS)$ sequence for the functional $J$ at a level $c$ and let $\varphi_1^{(n)}:=\varphi_1^{u_n}$ as in \eqref{Fam}. Then, $\|\varphi_1^{(n)}\|<\infty$ for any $n \in \mathbb{N}$ and there is a unique $w_n \in \partial J(u_n),$ which will be denoted by $J'(u_n),$ such that:
\begin{itemize}
		\item[$i)$] For some $\varphi_2^{(n)} \in \partial \Psi_2(u_n)$ one has
		$$J'(u_n)(v)= \langle \varphi_2^{(n)}, v \rangle+\langle\varphi_1^{(n)},v\rangle + \langle \Phi'(u_n), v\rangle ,\,\,\,\,\, \forall v \in H_{rad}^1(\mathbb{R}^N). $$
		
		\item[$ii)$] $J'(u_n)u_n=o_n(1)\|u_n\|$ with
		$$
		J'(u_n)(u_n) \leq \Psi_2^\circ(u_n,u_n) + \int_{\mathbb{R}^N}F_1'(u_n)u_n + \langle\Phi'(u_n), u_n \rangle ,\,\,\,\, \forall n \in \mathbb{N}.
		$$
\end{itemize}
	\end{lemma}
	
	\begin{proof}
		Let $(u_n)$ be a $(\rm PS)$ for the functional $J$. Then
		\begin{equation}\label{PS2}
		\Psi_2(v)-\Psi_2(u_n) + \int_{\mathbb{R}^N}(F_1(v)-F_1(u_n))\geq\langle-\Phi'(u_n),v-u_n\rangle +\langle w_n,v-u_n \rangle \,\,\,\, v \in H_{rad}^1(\mathbb{R}^N),
		\end{equation}
		with $w_n \in (H_{rad}^1(\mathbb{R}^N))'$, and $w_n \rightarrow 0$. Set $\phi \in C_{0,\,rad}^\infty(\mathbb{R}^N)$, and take $v:=u_n+t\phi$, with $t>0$. By Lemma \ref{G} it follows that
		\begin{equation}\label{Psi3}
		\Psi_2^\circ(u_n,\phi)+\int_{\mathbb{R}^N}F_1'(u_n)\phi \geq \langle -\Phi'(u_n), \phi \rangle + \langle w_n, \phi \rangle\,\,\,\,\, \forall \phi \in C_{0,\,rad}^\infty(\mathbb{R}^N),
		\end{equation}
		as $t\rightarrow 0^+$.
		Since
		$$ \langle \varphi_1^{(n)}, \phi \rangle = \int_{\mathbb{R}^N}F_1'(u_n)\phi\,\,\,\, \phi \in C_{0,\,rad}^\infty(\mathbb{R}^N),
		$$
		arguing as in the proof of Lemma \ref{CritSolution}, one has
		\begin{equation}\label{SupFam}
		\sup_{\phi \in C_{0,\,rad}^\infty(\mathbb{R}^N),\, \|\phi\|\leq 1}\langle \varphi_1^{(n)}\, \phi \rangle<\infty.
		\end{equation}
		Therefore, the functional $\varphi_1^n$ can be extended to the whole $H^1_{rad}(\mathbb{R}^N)$. By using (\ref{Psi3}), again as in Lemma \ref{CritSolution}, we get
		\begin{equation}
		-\Phi'(u_n)-\varphi_1^{(n)}+w_n \in \partial \Psi_2(u_n).
		\end{equation}
		Consequently, by setting $J'(u_n):=w_n$, one has
		\begin{equation}
		J'(u_n)=\varphi_2^{(n)}+\varphi_1^{(n)}+\Phi'(u_n),
		\end{equation}
		for some $\varphi_2^{(n)} \in \partial \Psi_2(u_n),$ so that $i)$ has been proved. In order to prove $ii)$, let us observe that
		$$J'(u_n)(u_n)=\langle w_n, u_n \rangle =o_n(1)\|u_n\|,$$
as $J'(u_n) \rightarrow 0$. Hence, by choosing $v:=u_n+tu_n$ in (\ref{PS2}), we have
		\begin{equation}\label{PS2a}
		J'(u_n)(u_n) \leq \frac{1}{t}[\Psi_2(u_n+tu_n)-\Psi_2(u_n)]+\int_{\mathbb{R}^N}\frac{1}{t}[F_1(u_n+tu_n)-F_1(u_n)] + \langle\Phi'(u_n),u_n\rangle.
		\end{equation}
		Since $F_1$ is convex, the map
		$$ t \longmapsto \frac{F_1(u_n+tu_n)-F_1(u_n)}{t},\,\,\, t>0 $$
		is monotone and
		$$ \frac{F_1(u_n+tu_n)-F_1(u_n)}{t}\rightarrow F_1'(u_n)u_n,$$
		as $t\rightarrow 0^+$. Now, Lemma \ref{Fam2} and (\ref{SupFam}) yields $F_1'(u_n)u_n \in L^1(\mathbb{R}^N)$ and  
		$$ \int_{\mathbb{R}^N} \frac{F_1(u_n+tu_n)-F_1(u_n)}{t} \rightarrow \int_{\mathbb{R}^N}F_1'(u_n)u_n,$$
by using the classical Lebesgue's Dominated Convergence Theorem.
		In conclusion, as $t \rightarrow 0$ in (\ref{PS2a}), by Lemma \ref{G}, it follows that
		$$ J'(u_n)(u_n) \leq \Psi_2^\circ(u_n,u_n) + \int_{\mathbb{R}^N}F_1'(u_n)u_n + \langle\Phi'(u_n), u_n \rangle.$$		
		This completes the proof.	
	\end{proof}

\noindent A consequence of Lemma \ref{Sub} is the following result that will be useful in order to prove that any $(\rm PS)$ sequence for the functional $J$ is bounded; see Lemma \ref{PSCB}.
	
	\begin{lemma}\label{L2Est}
		Let $(u_n)$ be a $(\rm PS)$ sequence for the functional $J$ at level $c$. Then
		\begin{equation}\label{L2estim}
		\int_{\mathbb{R}^N}|u_n|^2 \leq M+o_n(1)\|u_n\|,\,\,\, n\geq n_0
		\end{equation}
		for some $M>0$ and $n_0 \in \mathbb{N}$.
	\end{lemma}
	\begin{proof}
		Since $J(u_n) \rightarrow c$, there is $n_0 \in \mathbb{N}$ such that
		\begin{equation}\label{L1}
		J(u_n) \leq c+1,\,\,\,\, n\geq n_0.
		\end{equation}
		By setting $\tilde{J} = \tilde{I}|_{H_{rad}^1(\mathbb{R}^N)},$ i.e.,
		$$
		\tilde{J}(u)= \frac{1}{2}\|u\|^2 +\int_{\mathbb{R}^N}F_1(u)-\int_{\mathbb{R}^N}F_2(u)\,\,\,\, u \in H_{rad}^1(\mathbb{R}^N),
		$$
		we can write $J = \tilde{J}+\Psi_2|_{H_{rad}^1(\mathbb{R}^N)}$. By Lemmas \ref{Fam2} and \ref{Sub} Part - $ii)$, one has
		$$
		J'(u_n)(u_n)\leq \tilde{J}'(u_n)(u_n) + \Psi_2^\circ(u_n,u_n)
		$$
		as well as
		\begin{equation}\label{L2}
		J(u_n)-\frac{1}{2}J'(u_n)(u_n) \geq \frac{1}{2}\int_{\mathbb{R}^N}|u_n|^2+\left(\Psi_2(u_n)-\frac{1}{2}\Psi_2^\circ(u_n,u_n)\right).
		\end{equation}
	Now, gathering $J'(u_n)u_n=o_n(1)\|u_n\|$ with (\ref{L1}) and (\ref{L2}), we get
		$$
		c+1+o_n(1)\|u_n\| \geq \frac{1}{2}\int_{\mathbb{R}^N}|u_n|^2+\left(\Psi_2(u_n)-\frac{1}{2}\Psi_2^\circ(u_n,u_n)\right), \quad \forall n \geq n_0.
		$$
		In order to  finish the proof, it is enough to show that there is $M>0$ (independent of $n$) such that
		\begin{equation}\label{Ineq}
		\left(\Psi_2(u_n)-\frac{1}{2}\Psi_2^\circ(u_n,u_n)\right) \geq -M, \quad \forall n \in\mathbb{N}.
		\end{equation}
		Having this in mind, we employ the Lemma \ref{S3} to obtain
		$$
		\begin{aligned}
		\Psi_2(u_n)-\frac{1}{2}\Psi_2^\circ(u_n,u_n)= \int_{\mathbb{R}^N}F(x,u_n)-\frac{1}{2}\int_{\mathbb{R}^N}\eta^{(n)}u_n,
		\end{aligned}
		$$
		where $\eta^{(n)} \in L^2(\mathbb{R}^N)$ and $\eta^{(n)}(x) \in [\underline{f}(x,u_n(x)),\overline{f}(x,u_n(x))]$ a.e. in $\mathbb{R}^N$. The condition $(f_3)$ yields
		$$
		\begin{aligned}
		\int_{\mathbb{R}^N}F(x,u_n)-\frac{1}{2}\int_{\mathbb{R}^N}\eta^{(n)}u_n\geq -C\int_{\mathbb{R}^N}h(x) \geq - M,
		\end{aligned}
		$$
		for some $M=M_h>0$. This completes the proof.
	\end{proof}	
	
	\noindent Let us recall now the so-called logarithmic Sobolev inequality found in \cite[p. 144]{Alves-de Morais}, as well as \cite[Sentence (2.4)]{Ji-Szulkin} and the references therein:\par
\indent For each $b>0$,
	\begin{equation} \label{LogIneq}
	\int_{\mathbb{R}^N}u^2\log u^2 \leq \frac{b^2}{\pi}\|\nabla u\|_2^2+(\log \|u\|_2^2-N(1+\log b))\|u\|_2^2\,\,\,\, 
	\end{equation}
for every $u \in H^1(\mathbb{R}^N)$.\par

An immediate consequence of \eqref{LogIneq} is given below.
	
	\begin{corollary}\label{LogIneq1}
		There is $C>0$ such that
		$$
		\int_{\mathbb{R}^N}u^2\log u^2 \leq \frac{1}{2}\| \nabla u\|_2^2 + C(\log \|u\|_2^2)+1)\|u\|_2^2\,\,\,\, \forall u \in H^1(\mathbb{R}^N).
		$$
	\end{corollary}
	
	The following results involve $(\rm PS)$ sequences and will be used later on.

	\begin{lemma}\label{PSCB}
		If $(u_n)$ is a $(\rm PS)$ sequence for the functional $J$ at level $c \in \mathbb{R},$ then $(u_n)$ is bounded.
	\end{lemma}
	\begin{proof}
		By Lemma \ref{L2Est} and Corollary \ref{LogIneq1}, for each $r \in (0,1)$ there is $C_1>0$ such that
		$$
		\frac{1}{2}\int_{\mathbb{R}^N}u_n^2 \log u_n^2 \leq \frac{1}{4}\|u\|^2+C_1(1+\|u_n\|^{1+r}). 
		$$
 Since $J(u_n) \to c$, there is $n_0 \in \mathbb{N}$ such that
		$$
		\begin{aligned}
	c+1 \geq J(u_n) \geq \frac{1}{2}\|u_n\|^2-\frac{1}{2}\int_{\mathbb{R}^N}u_n^2 \log u_n^2,\,\,\,\, n\geq n_0.
		\end{aligned}
		$$
		Then
		$$
		c+1 \geq \frac{1}{4}\|u_n\|^2 - C_1(1+\|u_n\|^{1+r}),
		$$
for every $n\geq n_0$. This proves the desired result. 
	\end{proof}
	
	\begin{lemma}\label{PSC}
		The functional $J$ satisfies the $(\rm PS)$ condition.
	\end{lemma}	
	\begin{proof}
		Let $(u_n)$ be a $(\rm PS)$ sequence for $J$ at level $c$. By Lemma \ref{PSCB}, the sequence $(u_n)$ is bounded. Consequently, the embedding (\ref{Embedding}) yields
\begin{itemize}
\item[$i)$] $u_n \rightharpoonup u_0$ in $H_{rad}^1(\mathbb{R}^N);$
\item[$ii)$]  $u_n \rightarrow u_0 \in L^p(\mathbb{R}^N)$ with $p \in (2,2^*);$
\item[$iii)$]  $\|u_n\|\rightarrow M$ and $u_n(x) \rightarrow u_0(x)$ a.e. in $\mathbb{R}^N$.
		\end{itemize}
As $(u_n)$ is a $(\rm PS)$ sequence, we have that 
		\begin{equation}\label{PSs}
		\langle u_n, v-u_n \rangle + \Psi(v)-\Psi(u_n) \geq -\varepsilon_n\|v-u_n\| + \int_{\mathbb{R}^N}F_2'(u_n)(v-u_n)\,\,\,\, \forall v\in H_{rad}^1(\mathbb{R}^N),
		\end{equation}
		with $\varepsilon_n \rightarrow 0^+$. If we take $v:=u_0$ in (\ref{PSs}), the boundedness of $(u_n)$ and the  subcritical growth of $F_2$ immediately give
		\begin{equation}\label{pure}
		\langle u_n, u_0-u_n \rangle + \Psi(u_0)-\Psi(u_n) \geq o_n(1).
		\end{equation}
		Hence, the lower semicontinuity of $\Psi$ combined with the inequality \eqref{pure} leads to
		\begin{equation}
		{\|u_0\|}^2 \geq \lim \|u_n\|^2 =M^2,
		\end{equation}
on account of $i),ii)$ and $iii)$.
		In conclusion $u_n \rightarrow u_0$ in $H^1_{rad}(\mathbb{R}^N)$.
	\end{proof}
	
 In order to prove that $J$ satisfies the hypotheses of the Fountain Theorem \ref{Fountain1}, a suitable splitting of the Sobolev space $H_{rad}^1(\mathbb{R}^N)$ is necessary.
To this aim, we first observe that by \cite[Proposition 1.a.9 and Section 1.b, p. 8]{Lindestrauss} and \cite[Section 5]{Ji-Szulkin} the next property holds.

	\begin{lemma}\label{abstH}
		Let $A$ be a dense subset of $H^1(\mathbb{R}^N),$ then $H^1(\mathbb{R}^N)$ has an orthonormal hilbertian basis that is constituted by elements of $A$.
	\end{lemma}
		
	Thanks to Lemma \ref{abstH} the following result holds. 	
	\begin{corollary}
The space
		$H^1(\mathbb{R}^N)$ has an orthonormal hilbertian basis constituted by elements of $C_0^\infty(\mathbb{R}^N)$. Consequently, there exists a sequence $(v_j) \subset C_0^\infty(\mathbb{R}^N)$ such that
		\begin{equation}
		H^1(\mathbb{R}^N)=\overline{\bigoplus_{j \in \mathbb{N}}X_j}\quad\mbox{with}\quad X_j=\text{span}\,\{v_j\},
		\end{equation}
		and $\langle v_i, v_j \rangle = 0,$ for every $i \neq j$.\par
 Moreover, the same conclusion holds if we replace $H^1(\mathbb{R}^N)$ and $C_0^\infty(\mathbb{R}^N)$ by $H_{rad}^1(\mathbb{R}^N)$ and   $C_{0,\,rad}^\infty(\mathbb{R}^N)$ respectively.
	\end{corollary}
	
	From now on, let us consider
	\begin{equation}
	H_{rad}^1(\mathbb{R}^N)=\overline{\bigoplus_{j \in \mathbb{N}}X_j}
	\end{equation}
	and set
	\begin{equation}\label{DEC}
	Y_k := \bigoplus_{j=1}^k X_j\quad \mbox{ as well as } \quad Z_k:= \overline{\bigoplus_{j=k}^\infty X_j},
	\end{equation}
for every $k\in \N$.\par
Since the action of $\mathbb{Z}_2$ on $H_{rad}^1(\mathbb{R}^N)$ satisfies $(G_0)$ with $X_j \cong \mathbb{R}=:V$ we only need to prove that the functional $J$ satisfies the conditions $i)$ and $ii)$ of Fountain Theorem \ref{Fountain1}. \par
To this aim, let us will apply the following lemma 
	
	\begin{lemma}\label{Convergence}
		Let $\beta_k$ defined by
		\begin{equation}
		\beta_k:=\sup_{u \in Z_k, \|u\|=1} \|u\|_p.
		\end{equation}
		Then $\beta_k \rightarrow 0$.	
	\end{lemma}
\begin{proof} See \cite[Lemma 3.8]{Willem} and also the proof of Proposition 3.7 in \cite{Ji-Szulkin} for additional comments and remarks.
\end{proof}

Taking into account Lemma \ref{Convergence}, we are able to prove that the functional $J$ satisfies the Fountain geometry. 
	\begin{lemma}\label{Fino}
		The functional $J$ verifies 
\begin{itemize}
	\item [$i)$] $\displaystyle{\sup_{u \in Y_k, \|u\|= \rho_k}}J(u)\leq 0;$
	\item [$ii)$] $	\displaystyle{\inf_{u \in Z_k, \|u\|=r_k}}J(u) \rightarrow \infty.$
\end{itemize}
	\end{lemma}
	\begin{proof}
 We first recall that
		$$J(u)=\frac{1}{2}\|u\|^2+\int_{\mathbb{R}^N}F(x,u)+\int_{\mathbb{R}^N}F_1(u) - \int_{\mathbb{R}^N} F_2(u)\,\,\,\,u\in H_{rad}^1(\mathbb{R}^N).$$
		\indent Part $i)$ - By $(f_1)$ one has
		$$|F(x,s)|\leq B|s|^2,\,\,\, \forall x \in \mathbb{R}^N \quad \mbox{and} \quad \forall s \in \mathbb{R},$$
		for some constant $B>0$. Now, by definition, since $Y_k \subset C_{0,\,rad}^\infty(\mathbb{R}^N)$ it follows that $Y_k \subset D(J)$ for each $k \in \mathbb{N}$. Hence
		\begin{equation}
		J(u) \leq \frac{1}{2}\|u\|^2+B\|u\|_2^2-\frac{1}{2}\int_{\mathbb{R}^N}u^2\log u^2,
		\end{equation}
for every $u\in Y_k$.\par
		\noindent If we take $v:=\displaystyle{\frac{u}{\|u\|}}$ for $u \neq 0$, it follows that
		\begin{equation}\label{Estimate}
		\begin{aligned}
		J(u)&\leq \frac{1}{2}\|u\|^2\left(1+B-\int_{\mathbb{R}^N}v^2\log(v^2\|u\|^2)\right)\\
		&= \frac{1}{2}\|u\|^2\left(1+B-\int_{\mathbb{R}^N}v^2\log v^2-\log (\|u\|^2)\int_{\mathbb{R}^N}v^2\right),
		\end{aligned}
		\end{equation}
		for every $u\in Y_k$. As $\dim Y_k < \infty$,  all the norms on $Y_k$ are equivalent. Consequently, if $\|u\| = \rho_k\approx \infty$, one gets
	$$
1+B-\int_{\mathbb{R}^N}v^2\log v^2-\log (\|u\|^2)\int_{\mathbb{R}^N}v^2\leq 0.$$ 
Hence
		\begin{equation*}
		\sup_{u \in Y_k, \|u\|= \rho_k}\,J(u) \leq 0,
		\end{equation*}
so that $i)$ is verified.\par
		\indent Part $ii)$ - By $(A_2)$ for every $s\in \R$,
		$$|F_2(s)|\leq C|s|^p,\quad p\in(2,2^*),$$
		for some $C>0$. Hence
	$$
	J(u)\geq \frac{1}{2}\|u\|^2-\int_{\mathbb{R}^N}F_2(u)\geq \frac{1}{2}\|u\|^2-\beta_k^p C\|u\|^p,
	$$
	for every $u\in Z_k$. Moreover, by Lemma \ref{Convergence} one has $\beta_k\rightarrow 0$. Then, by choosing 			$$r_k:=(pC\beta_k^p)^{\frac{1}{2-p}},$$ it follows that $r_k \rightarrow \infty$ and
		\begin{equation*}
		J(u)\geq\left(\frac{1}{2}-\frac{1}{p}\right)r_k^2.
		\end{equation*}
	In conclusion
		\begin{equation*}
		\inf_{u \in Z_k,\,\|u\|=r_k}\,J(u) >0,
		\end{equation*}
	for $k$ sufficiently large.
	\end{proof}
	
	\noindent {\bf Conclusion of proof of Theorem \ref{Theorem}.} First of all, we emphasize that the minimax levels
	\begin{equation*}
	c_k:=\inf_{\gamma \in \Theta_k}\,\sup_{u\in B_k} J(\gamma(u))
	\end{equation*}
	are finite. Indeed, if we take $\tilde{\gamma}:=Id|_{B_k},$ by using the classical inequality
	\begin{equation*}
	|t^2\log t^2| \leq C(|t|+|t|^p),\,\,\, p>2 \quad \mbox{and} \quad \forall t \in \mathbb{R},
	\end{equation*}
	we infer that there exists $C_1>0$ such that
	\begin{equation}\label{Estimate2}
	J(\tilde{\gamma}(u))\leq |J(u)|\leq \frac{1}{2}\|u\|^2+ B\|u\|_2^2+C_1(\|u\|_1+\|u\|_p^p),
	\end{equation}
for every $u \in B_k \subset Y_k$.
 The equivalence of the norms in $Y_k$ combined with (\ref{Estimate2}) guarantee that 	
	$$
	c_k=\inf_{\gamma \in \Theta_k}\,\sup_{u\in B_k} J(\gamma(u))\leq \sup_{u\in B_k}\,J(\tilde{\gamma}(u)) <\infty.
	$$

\indent Finally, we would like to point out that  if $u \in H^{1}(\mathbb{R}^N)$ is a critical point of $I$, then there exists $\rho \in L^{2}(\mathbb{R}^N)$ with $$ \rho(x) \in [\underline{f}(x,u(x)),\overline{f}(x,u(x))]\,\,\,\, \text{a.e. in}\,\,\, \mathbb{R}^N,$$ such that
$$
\int_{\mathbb{R}^N}(\nabla u \nabla \phi +u\phi)\,dx+\int_{\mathbb{R}^N}\rho(x)\phi\,dx=\int_{\mathbb{R}^N}u^{2}\log u \phi \,dx \quad \forall \phi \in C_{0}^{\infty}(\mathbb{R}^N).
$$	
Therefore, by elliptic regularity theory, there is 	$r\geq 1$ such that $u \in H^1(\mathbb{R}^N)\cap W_{loc}^{2,r}(\mathbb{R}^N)$ and
$$
-\Delta u +u +\rho(x)= u\log u^2\,\,\,\, \text{a.e. in}\,\,\, \mathbb{R}^N.
$$
In conclusion
\begin{equation*}
\Delta u - u + u\log u^2 \in [\underline{f}(x,u(x)),\overline{f}(x,u(x))]\,\,\,\, \text{a.e. in}\,\,\, \mathbb{R}^N.
\end{equation*}

\subsection{A concave perturbation of logarithmic equation} In this subsection we study the existence of solutions for the following class of problems
$$
\left\{\begin{aligned}
-&\Delta u + u = u\log u^2 + \lambda h(x)|u|^{q-2}u,  \;\;\mbox{in}\;\;\mathbb{R}^{N}, \\
& u \in H^{1}(\mathbb{R}^{N}),
\end{aligned}
\right. \leqno{(P_2)}
$$
where $\lambda$ is a positive parameter, $q\in(1,2)$ and $h:\mathbb{R}^N\rightarrow\mathbb{R}$ satisfies $(h_1)$. By using the same notations of the previous subsection, the energy functional associated to $(P_2)$ is given by
\begin{equation}\label{EF1}
I_\lambda(u):= \frac{1}{2}\|u\|^2+\int_{\mathbb{R}^N}F_1(u)-\int_{\mathbb{R}^N}F_2(u)-\frac{\lambda}{q}\int_{\mathbb{R}^N}|u|^qh(x),\,\,\,\, u \in H^1(\mathbb{R}^N).
\end{equation}

\noindent Note that $I_\lambda \in (H_0)$, with $I_\lambda(u)=\Phi(u)+\Psi(u),$ where
$$\Phi(u):= \frac{1}{2}\|u\|^2-\int_{\mathbb{R}^N}F_2(u)-\frac{\lambda}{q}\int_{\mathbb{R}^N}h|u|^q
$$
and
$$\Psi(u):=\int_{\mathbb{R}^N}F_1(u).$$

In the sequel, we say that a function $u \in H^1(\mathbb{R}^N)$ is a solution of $(P_2)$ if $u^2\log u^2 \in L^1(\mathbb{R}^N)$ and
\begin{equation}\label{AS2}
\int_{\mathbb{R}^N}(\nabla u \nabla \phi+u\phi)=\int_{\mathbb{R}^N}(u\log u^2 \phi+\lambda h(x)|u|^{q-2}u\phi),\,\,\,\,\forall \phi \in C_0^\infty(\mathbb{R}^N).
\end{equation}
\indent By Part - $ii)$ of Lemma \ref{Psi1} it is to see that any critical point of $I_\lambda \in (H_0)$ is a solution of $(P_2)$; see also \cite[Lemma 2.1]{Alves-de Morais}. Moreover, again by Theorem \ref{PC}, if we define $J_\lambda:=I_\lambda|_{H_{rad}^1(\mathbb{R}^N)}$, the critical points of $J_\lambda$ are also critical points of the functional $I_\lambda$. \par

The main result this subsection reads as follows.

\begin{theorem}\label{CL}
	The functional $J_\lambda$ has infinitely many critical points $(u_n)$ with $J_\lambda(u_n)\rightarrow 0$ as $n\rightarrow \infty$. Hence, problem $(P_2)$ has infinitely many nontrivial solutions.
\end{theorem}

In order to prove Theorem \ref{CL}, let us introduce a modified functional $\tilde{J}_\lambda$ which will be crucial in our approach. However, let us start by proving the following technical result.

\begin{proposition}
	 If $\lambda \approx 0^+,$ then there is a function
	$$g(t):= \frac{1}{2}t^2-Bt^p-C\lambda t^q,\,\,\,t>0,$$
with $p \in (2,2^*)$  and  $B,C>0$, that attains a non-negative maximum and
	$$J_\lambda(u)\geq g(\|u\|)\quad \forall u \in H^{1}(\mathbb{R}^N).$$
\end{proposition}
\begin{proof}
Since $F_1\geq 0$, we easily get
\begin{equation*}
\begin{split}
 J_\lambda(u)\geq & \frac{1}{2}\|u\|^2-\int_{\mathbb{R}^N}F_2(u)-\frac{\lambda}{q}\int_{\mathbb{R}^N}|u|^qh(x) \\
 &\geq \frac{1}{2}\|u\|^2- C_1\|u\|^p-\lambda C_2\|u\|^q\\
 & =:g(\|u\|),
\end{split}
\end{equation*}
	for some $C_1=C(p)>0$ and $C_2=C(h,q)>0$. Moreover, if $\lambda \approx 0^+$ it is clearly seen that the function $g$ attains a non-negative maximum.
\end{proof}

\noindent Now, fix $R_0, R_1$ and $R_2$ positive constants satisfying: 
\begin{itemize}
\item[$(g_1)$] $g|_{[0,R_0]}\leq 0$ and $g(R_0)=0$;
\item[$(g_2)$] $g|_{[R_0, R_2]}\geq 0$, $g|_{[R_2,\infty)}\leq 0$ and $g(R_2)=0$, where $R_0<R_1<R_2$ and $R_1$ is the point in which $g$ attains its maximum value; note that $g(t)\rightarrow -\infty$, as $t\rightarrow \infty$.
\end{itemize}
Moreover, take $\eta \in C^\infty([0,\infty))$ such that the following condition holds:
\begin{itemize}
\item[$(\eta_1)$] $\eta$ is a non-negative and non-increasing function such that $$\eta|_{[0,R_0]}\equiv 1\quad \mbox{ and }\quad\eta|_{[R_2,\infty)}\equiv 0.$$
\end{itemize}
Set $\varphi(u):=\eta(\|u\|)$. Arguing as in \cite{peral}, let us consider the energy functional
\begin{equation}\label{EF2}
\tilde{J}_\lambda(u):=\frac{1}{2}\|u\|^2+\int_{\mathbb{R}^N}F_1(u)-\varphi(u)\int_{\mathbb{R}^N}F_2(u)-\frac{\lambda}{q}\int_{\mathbb{R}^N}|u|^qh(x),
\end{equation}
for every $u\in H_{rad}^1(\mathbb{R}^N)$.\par

\begin{lemma}\label{TL2}
Let $\tilde{J}_\lambda$ be the functional given in \eqref{EF2}. Then, the following facts hold:
\begin{itemize}
	\item[$i)$] $\tilde{J}_\lambda \in (H_0)$ with $\tilde{J}_\lambda=\tilde{\Phi}_\lambda+\tilde{\Psi}$ and $\tilde{\Psi}=\Psi|_{H_{rad}^1(\mathbb{R}^N)}$$;$
	\item[$ii)$] If $\tilde{J}_\lambda(u)<0$ then $\|u\|< R_0$ and
	$\tilde{J}_\lambda(u)=J_\lambda(u);$
	\item[$iii)$] Let $(u_n)$ be a ${(\rm PS)}_{c}$ sequence for $\tilde{J}_\lambda$ with $c<0$ then $(u_n)$ is a ${(\rm PS)}_c$ sequence for $J_\lambda;$
	\item[$iv)$] If $u \in B_{R_0}(0)$ is a critical point of $\tilde{J}_\lambda$ then $u$ is a critical point of $J_\lambda$.
\end{itemize}
\end{lemma}

\begin{proof}
	Part $i)$ immediately follows by $(\eta_1)$ and the definition of $\tilde{J}_\lambda$.
    Moreover, if $\lambda\approx 0^+$ then
	$$\tilde{g}(t):= \frac{1}{2}t^2-\lambda C_2t^q\geq 0$$
	for every $t\geq R_2$ and $\tilde{J}_\lambda(\|u\|)\geq\tilde{g}(\|u\|)$. Hence, condition $ii)$ holds. The rest of the proof is an easy consequence of $i)$ and $ii)$.
\end{proof}

By using the above notations and results we are able to prove Theorem \ref{CL}.

\begin{proof}[Proof of Theorem \ref{CL}] -	
By Lemma \ref{TL2} it is sufficient to show that $\tilde{J}_\lambda$ has a sequence of critical points $(u_n)$ with $u_n \in B_{R_0}(0)$ for every $n \in \mathbb{N}$. This will be done by showing that $\tilde{J}_\lambda$ satisfies the hypotheses of Theorem \ref{HT}. To this aim, we first notice that $\tilde{J}_\lambda$ is even and $\tilde{J}_\lambda(0)=0$. Moreover $\tilde{J}_\lambda$ is coercive and consequently any ${(\rm PS)}_{c}$ sequence for $\tilde{J}_\lambda$ is bounded. If $(u_n)$ is a ${(\rm PS)}_{c}$ sequence for $\tilde{J}_\lambda$, with $c<0$, then Lemma \ref{TL2} ensures that $(u_n)$ is also a ${(\rm PS)}_{c}$ sequence for $J_\lambda$. Finally, arguing as in Lemma \ref{PSC}, it easily seen that $\tilde{J}_\lambda$ satisfies the ${(\rm PS)}_{c}$ condition for $c<0$. It remains to show that $\tilde{J}_\lambda$ satisfies $i)$ and $ii)$ of Theorem \ref{HT}.\\
	\indent Part $i)$ -  Since $\tilde{J}_\lambda$ satisfies
	$$\tilde{J}_\lambda(u)\geq g(\|u\|)\, \quad \forall u \in H^{1}(\mathbb{R}^N)$$
	and $\tilde{J}_\lambda(u)\geq 0$ for every $\|u\|\geq R_2$, we conclude that $\tilde{J}_\lambda$ is bounded from below. Consequently
	$$ c_j:=\inf_{A \in \Gamma_j}\sup_{u \in A}\tilde{J}_\lambda(u)>-\infty.$$
	\indent Part $ii)$ - For each $k \in \mathbb{N}$, let us consider $Y_k$ and $Z_k$ as in (\ref{DEC}). In this case $\text{dim}Y_k<\infty$ and $Y_k \subset C_0^\infty(\mathbb{R}^N)$. Bearing in mind that
	$$F_1(u)<\infty\,\,\,\forall u \in Y_k,$$
	we infer that $Y_k \subset D(\tilde{J}_\lambda)$ for any $k \in \mathbb{N}$. As $\tilde{J}_\lambda\equiv J_\lambda$ in $B_{R_0}$, one has
	$$\tilde{J}_\lambda(u)=\frac{1}{2}\|u\|^2-\frac{1}{2}\int_{\mathbb{R}^N}u^2\log u^2-\frac{\lambda}{q}\int_{\mathbb{R}^N}|u|^qh(x).$$
Moreover, if $\delta\approx 0^+$
	$$|t|^2|\log t^2|\leq C_1(|t|^{2-\delta}+|t|^{2+\delta})\quad \forall t \in\mathbb{R}$$
	for some $C_1=C_1(\delta)>0$. Consequently
	$$\tilde{J}_\lambda(u)\leq\frac{1}{2}\|u\|^2+C\int_{\mathbb{R}^N}(|u|^{2-\delta}+|u|^{2+\delta})-\frac{\lambda}{q}\int_{\mathbb{R}^N}|u|^qh(x),$$
	for every $u \in B_{R_0}$. 
Now, if $u \in Y_k$ then $u \in L^r(\mathbb{R}^N)$ for every $r \in [1,2)$. Since all the norms on $Y_k$ are equivalent, one has
	\begin{equation}\label{IJ}
	\tilde{J}_\lambda(u)\leq\frac{1}{2}\|u\|^2+C_2(\|u\|^{2-\delta}+\|u\|^{2+\delta})-C\|u\|^q.
	\end{equation}
		for some constant $C_2>0$. Now, for each $k \in \mathbb{N}$, fix $A:=S_\rho(0)\cap Y_k$ with $\rho \approx0^+$. Then $A \in \Sigma$ and $\gamma(A)=k$. By choosing $\delta$ such that $2-\delta>q$, by (\ref{IJ}) it follows that
	$$ \sup_{u \in A}\tilde{J}_\lambda(u)<0,$$
finishing the proof. 
\end{proof}

\subsection{A problem involving the 1-Laplacian operator with subcritical growth}

	In this subsection we study the existence of infinitely many solutions for the following problem
	$$
	\left\{\begin{aligned}
	-&\Delta_1 u = |u|^{p-2}u,  \;\;\mbox{in}\;\;\Omega, \\
	& u|_{\partial\Omega}=0, \;\;\mbox{on}\;\; \partial \Omega,
	\end{aligned}
	\right. \leqno{(P_3)}
	$$
	where $\Omega \subset \mathbb{R}^N$ (with $N\geq 2$) is a bounded domain with smooth boundary $\partial \Omega$ and $p\in (1,1^*)$. In order to simplify the notation, we set $q:=p/(p-1)$.
	
	From now on we denote by $\mathcal{M}(\Omega, \mathbb{R}^N)$ (briefly $\mathcal{M}(\Omega)$) the space of the vectorial Radon measures on $\Omega$ and by $BV(\Omega)$ the space of the functions $u:\Omega\rightarrow\mathbb{R}$ of bounded variation, that is, 
	$$
	BV(\Omega):=\{u \in L^1(\Omega):\,\,Du \in \mathcal{M}(\Omega)\},
	$$
	where $Du$ denotes the distributional derivative of $u$. It is well known that $u \in BV(\Omega)$ if and only if $u \in L^1(\Omega)$ and
	$$
	\int_{\Omega}|Du|=\sup\,\left\{\int_{\Omega}u\text{div}\phi:\,\phi\in C^1_0(\Omega,\mathbb{R}^N),\,\,\mbox{ and }\,\,\|\phi\|_\infty\leq 1)\right\}<\infty.
	$$
	Moreover $BV(\Omega)$ is a Banach space endowed with the norm
	$$\|u\|_{BV(\Omega)}:=\int_{\Omega}|Du|+\int_{\partial \Omega}|u|d\mathcal{H}^{N-1},$$
	where, as usual, $\mathcal{H}^{N-1}$ denotes the $(N-1)$-dimensional Hausdorff measure. We also recall that the continuous embedding
	\begin{equation}\label{EMB}
	BV(\Omega)\hookrightarrow L^r(\Omega),\,\,\, r\in [1,1^*]
	\end{equation}
	is compact provided that $r\in[1,1^*)$; see \cite{Ambrosio, Attouch, Kawohl} for advanced theoretical results on the subject.
	
According to Kawohl-Schuricht \cite{Kawohl} and Degiovanni \cite{Degiovanni} the notion of solution for problem $(P_3)$
can be formulated as follows.

	\begin{definition}\label{SP3}
	We say that a function $u \in BV(\Omega)$ is a solution of $(P_3)$ if there exists $z\in L^\infty(\Omega, \mathbb{R}^N)$ with $\|z\|_\infty\leq 1,$ such that
	\begin{equation*}
		\left\{\begin{aligned}
		 &-\int_\Omega u{\rm div} z =\int_{\Omega}|Du|+\int_{\partial \Omega}|u|d\mathcal{H}^{N-1},\,\,\,\,{\rm div} z\in L^q(\Omega),\\
		 &- {\rm div} z = |u|^{p-2}u\,\,\,\,\text{a.e. in}\,\,\ \Omega,
		\end{aligned}\right.
	\end{equation*}
	where $q:=p/(p-1)$.
	\end{definition}
	
	Now, let us consider the energy functional $I:L^p(\Omega)\rightarrow (-\infty,\infty]$ given by 
	\begin{equation} \label{NOVOFUNCT}
	I(u)=\Phi(u) +\Psi(u),
	\end{equation}
 where
	$$\Phi(u):=-\frac{1}{p}\int_{\Omega}|u|^p$$
	and
	$$\Psi(u):=\left\{\begin{aligned}
	&\int_{\Omega}|Du|+\int_{\partial \Omega}|u|d\mathcal{H}^{N-1}\,\,\, &u \in BV(\Omega)\\
	&\quad \quad \infty\,\,\,\,& u\in L^p(\Omega)\setminus BV(\Omega)
	\end{aligned}\right.,
	$$
for every $u\in L^p(\Omega)$.\par
It is easily seen that $\Phi \in C^1(L^p(\Omega),\mathbb{R})$ as well as $\Psi$ is a convex and l.s.c. functional, so that $I\in (H_0)$. Consequently $D(I)=BV(\Omega)$ and for each fixed $u \in BV(\Omega),$ the subdifferential $\partial \Psi(u)$ can be identified as a subset of $L^q(\Omega)$.\par
 The next results will be crucial in the sequel.
	
	\begin{lemma}\label{LE1}
		If $u \in BV(\Omega)$ and $\partial \Psi(u)\neq \emptyset$ then $u \in L^\infty(\Omega)$.
	\end{lemma}

	\begin{proof}
		We first notice that $L^{1^*}(\Omega) \hookrightarrow L^p(\Omega),$ so that $L^q(\Omega)\hookrightarrow L^N(\Omega)$. Consequently, if \linebreak $w \in \partial \Psi(u) \subset L^q(\Omega)$, one has $w \in L^N(\Omega)$. The conclusion is achieved arguing as in \cite[Proposition 3.3]{Degiovanni}.
	\end{proof}	

	\begin{lemma}\label{LE2}
	If $u \in BV(\Omega)$ then, for each $w \in \partial \Psi(u) ,$ there exists $z \in L^\infty(\Omega,\mathbb{R}^N)$ with $\|z\|_\infty \leq 1,$ such that
		$$
		\left\{\begin{aligned}
		&w = - \text{\rm div} z \in L^q(\Omega)\\
		-&\int_{\Omega} u \text{\rm div} z=\int_{\Omega}|Du|+\int_{\partial \Omega}|u|d\mathcal{H}^{N-1}.
		\end{aligned}\right.
		$$
	\end{lemma}
	\begin{proof}
		Let us define
		$$
		\tilde{\Psi}(u):=\left\{\begin{aligned}
		&\int_{\Omega}|Du|+\int_{\partial \Omega}|u|d\mathcal{H}^{N-1}\,\,\,\, &u \in BV(\Omega)\\
		&\quad \quad \infty\,\,\,\,\,& u\in L^{1^*}(\Omega)\setminus BV(\Omega)
		\end{aligned}\right.,
		$$	
and take $w \in \partial \Psi(u) \subset L^q(\Omega)$. Then $ w\in L^N(\Omega)$ and
		$$\begin{aligned}
		\tilde{\Psi}(v)-\tilde{\Psi}(u)=\Psi(v)-\Psi(u) \geq\int_{\Omega}w(v-u)\,\,\, \quad \forall v\in BV(\Omega)=D(\tilde{\Psi}),
		\end{aligned}
		$$
		so that $w \in \partial \tilde{\Psi}(u)$. The conclusion follows by \cite[Proposition 4.23]{Kawohl}.
	\end{proof}

	The next result connects critical points of the energy functional $I$ with solutions of $(P_3)$.
	
	\begin{lemma}
		If $u \in BV(\Omega)$ is a critical point of the functional $I$ then $u \in L^\infty(\Omega)$. Moreover, the function $u$ is a solution of $(P_3)$ in the sense of Definition \ref{SP3}.
	\end{lemma}
	\begin{proof}
		Let $u \in BV(\Omega)$ be a critical point of $I$. Then
		$$-\Phi'(u)\in \partial \Psi(u)\subset L^q(\Omega).$$
		Thereby, there exists $w \in \partial \Psi(u)$ such that
		$$ -\Phi'(u)= w\,\,\,\, \text{in}\,\,\,L^q(\Omega).$$
		
	\noindent Consequently, Lemma \ref{LE2} and the definition of $\Phi$ yield the existence of $z \in L^\infty(\Omega, \mathbb{R}^N)$, with $\|z\|_\infty\leq 1$, such that $- {\rm div} z = w$ in $L^q(\Omega)$ and
		$$
		\left\{\begin{aligned}
		&-\int_\Omega u\text{\rm div} z =\int_{\Omega}|Du|+\int_{\partial \Omega}|u|d\mathcal{H}^{N-1},\,\,\,\text{\rm div} z\in L^q(\Omega)\\
		&- \text{\rm div} z = |u|^{p-2}u\,\,\,\text{a.e. in}\,\,\ \Omega.
		\end{aligned}\right.
		$$
	Moreover, Lemma \ref{LE1} ensures that $u \in L^\infty(\Omega)$ and the proof is over. 
	\end{proof}
	
	By using the above notations and results we are able to prove the main result of this subsection.
	
	\begin{theorem}\label{ultimo}
		The functional $I$ has infinitely many critical points $(u_n)$ with $I(u_n) \rightarrow \infty$ as $n\rightarrow \infty$. Hence, problem $(P_3)$ has infinitely many nontrivial solutions. 
	\end{theorem}

	\begin{proof}
	Hereafter, we are going to prove that $I$ verifies the assumptions of Theorem \ref{SP} with $Y=\{0\}$. We first prove that $I$ satisfies the compactness $(\rm PS)$ condition. To this end, let $(u_n)$ be a $(\rm PS)$ sequence for $I$. So, for some $c \in \mathbb{R}$, one has
		$$I(u_n)\rightarrow c$$
		and
		$$\Psi(v)-\Psi(u_n)\geq \int_{\Omega}|u_n|^{p-2}u_n(v-u_n)+\int_{\Omega}w_n(v-u_n),\,\quad \forall v \in BV(\Omega),$$
		where $w_n \in L^q(\Omega)$ and $w_n\rightarrow 0$ in $L^q(\Omega)$. The last inequality gives
		$$|u_n|^{p-2}u_n +w_n \in\partial \Psi(u_n),\,\,\,\,\forall n \in \mathbb{N}.$$
		Hence, Lemma \ref{LE2} yields
		$$\Psi(u_n)=\int_{\Omega}|Du_n|+\int_{\partial \Omega}|u_n|d\mathcal{H}^{N-1}=\int_{\Omega}|u_n|^p+\int_{\Omega}w_nu_n, \quad \forall n \in \N.$$
If we set
		$$A(u_n):=\Psi(u_n)-\int_{\Omega}|u_n|^p+\int_{\Omega}w_nu_n=0,$$
	the classical Hölder's inequality leads to
		$$\begin{aligned}
		c+1&\geq I(u_n)-\frac{1}{r}A(u_n)\\
		&\geq\left(1-\frac{1}{r}\right)\Psi(u_n)+\left(\frac{1}{r}-\frac{1}{p}\right)\|u_n\|^p_{L^p(\Omega)}-\frac{1}{r}\|w_n\|_{L^q(\Omega)}\|u_n\|_{L^p(\Omega)}\\
		&\geq C_1\|u_n\|_{BV(\Omega)}+C_2\left(\|u_n\|_{L^p(\Omega)}^p-\|u_n\|_{L^p(\Omega)}\right),
		\end{aligned}
	$$
 for some $r<p$ and $n$ large enough. Since the real function $h(t):=t^p-t$, for every $t\geq 0$, is bounded from below, the last inequality clearly implies  $\sup_{n \in \mathbb{N}}\|u_n\|_{BV(\Omega)}<\infty.$ Therefore the $(\rm PS)$ condition is verified, since the embedding $BV(\Omega)\hookrightarrow L^p(\Omega)$ is compact. Now, if $u \in BV(\Omega)$ is a critical point of $I$ then
		$$|u|^{p-2}u \in \partial \Psi(u).$$
Consequently, by Lemma \ref{LE2},
		$$\int_\Omega |u|^p =\int_{\Omega}|Du|+\int_{\partial \Omega}|u|d\mathcal{H}^{N-1}.$$
		Thereby, by setting
		$$ B(u)= \int_{\Omega}|Du|+\int_{\partial \Omega}|u|d\mathcal{H}^{N-1}-\int_\Omega |u|^p,$$
		one has
		$$I(u)=I(u)-\frac{1}{p}B(u)=\left(1-\frac{1}{p}\right)\|u\|_{BV(\Omega)}\geq0,$$
for every $u\in L^p(\Omega)$. Hence, the set $I^{-c}$ has no critical points for any $c>0$.\par
 \noindent Finally, let us prove that the functional $I$ satisfies conditions $i)$ and $ii)$ of Theorem \ref{SP}.\\
		\indent Part $i)$ - Without loss of generality we can suppose $u\in BV(\Omega)$, otherwise $I(u)=\infty$. Now, if $u \in BV(\Omega)$, the embedding $BV(\Omega)\hookrightarrow L^p(\Omega)$ immediately yields
		$$I(u)\geq C\|u\|_{L^p(\Omega)}-\frac{1}{p}\|u\|_{L^p(\Omega)}^p,$$
		for some constant $C>0$. Since $p>1$, if $\|u\|_{L^p(\Omega)}=r\approx0^+$, we also have  
		$$I(u)\geq \rho,$$
	for some $\rho>0$. Thus, condition $i)$ of Theorem \ref{SP} is proved with $Z=L^p(\Omega)$.\\
		\indent Part $ii)$ - For each $k \in \mathbb{N}$, let us consider $X_k$ be a $k$-dimensional subspace of $C_0^\infty(\Omega)$. Since all the norms are equivalent on $X_k$, it easily seen that
		$$I(u)\leq C_k\|u\|_{L^p(\Omega)}-\frac{1}{p}\|u\|_{L^p(\Omega)}^p\,\,\,\,\forall u \in X_k,$$
		for a convenient $C_k>0$. Thus
		$$I(u)\rightarrow -\infty,\,\,\,\text{as}\,\,\,\|u\|_{L^p(\Omega)}\rightarrow \infty\, \text{and}\,u \in X_k,$$
finishing the proof. 
	\end{proof}	

\vspace{0.2 cm}

\noindent {\bf Acknowledgments:} The authors would like to thank the anonymous referee by his/her valuable suggestions, which were very important to improve the paper.

	

\begin{thebibliography}{1}
	
	
	
	
	
	
	
	
	
	\bibitem{Alves-de Morais} Alves, C.O., de Morais Filho, D.C.: \textit{Existence and concentration of positive solutions for a Schrödinger logarithmic equation}, Z. Angew. Math. Phys. 69, 144-165, (2018)
	
	\bibitem{Alves-Fig-Pimenta} Alves, C.O., Figueiredo, G.M.,  Pimenta, M.T.O.: \textit{Existence and profile of ground-state
	solutions to a 1-Laplacian problem in $\mathbb{R}^N$}, Bull. Braz. Math. Soc., New Series DOI 10.1007/s00574-019-00179-4
	
	\bibitem{AlvesJVJA2}  Alves, C.O.,  Gon\c calves, J.V., Santos, J.A.: {\it Strongly nonlinear multivalued elliptic equations on a bounded domain,} J. Glob. Optim. 58, 565-593, (2014)
	
	\bibitem{Alves-Ji} Alves, C.O., Ji, C.: \textit{Existence and concentration of positive solutions for a logarithmic Schrödinger equation via penalization method}. Calc. Var. 59, 21 (2020). https://doi.org/10.1007/s00526-019-1674-1

	\bibitem{Alves-Pimenta} Alves, C.O., Pimenta, M.T.O.: \textit{On existence and concentration of solutions to a class of quasilinear problems involving the 1-Laplace operator}, Calc. Var. 56, 143 (2017). https://doi.org/10.1007/s00526-017-1236-3
	
	\bibitem{Ambrosio} Ambrosio, L., Fusco, N., Pallara, D.: \textit{Functions of Bounded Variation and Free
	Discontinuity Problems}, Oxford University Press, Oxford, (2000).

	\bibitem{Attouch} Attouch, H., Buttazzo, G., Michaille, G.: \textit{Variational Analysis in Sobolev and BV Spaces: Applications
	to PDEs and Optimization}, MPS-SIAM, Philadelphia (2006)
	
	
	
	
	
	
	
	
	
	
	
	
	
	\bibitem{Bartsch0} Bartsch, T.: \textit{Infinitely many solutions of a symmetric Dirichlet problem}, Nonlinear Anal.TMA 20, 1205-1216, (1993)
	
	\bibitem{Bartsch}  Bartsch, T.: \textit{Topological Methods for Variational Problems with Symmetries}, Lectures Notes in Math. 1560, Springer, (1993)
	
	\bibitem{Bartsch1} Bartsch, T., Willem, M.: \textit{On an elliptic equation with concave and convex nonlinearities}, Proc. Amer. Math. Soc. 123, 3555-3561, (1995)
	
	\bibitem{Batkam} Batkam, C. J., Colin F.: \textit{Generalized Fountain Theorem and applications to strongly indefinite semilinear problems}, J. Math. Anal. Appl. 405, 438-452, (2013)

    \bibitem{BJ} Bereanu, C., Jebelean, P.: \textit{Multiple critical points for a class of periodic lower semicontinuous functionals}, Discrete Contin. Dyn. Syst. 33, 47-66, (2013)


	
	
	
	
	\bibitem{Motreanu1} Carl, S., Le, V. K., Motreanu, D.: \textit{Nonsmooth Variational Problems and Their Inequalities: Comparison Principles and Applications}, Springer, New York, (2007)
	
	
	\bibitem{Chang1} Chang, K. C.: {\it Variational methods for nondifferentiable functionals and their applications to partial differential equations}. J. Math. Anal. Appl. 80, 102-129 (1981)

	\bibitem{Chang2} Chang, K.: The spectrum of the 1-Laplace operator. Commun. Contemp. Math. 9(4), 515–543 (2009)
	
	\bibitem{Clarke} Clarke, F. H.: {\it Generalized gradients and applications,} Trans. Amer. Math. Soc. 205, 247-262, (1975)
	
	\bibitem{Clarke1} Clarke, F. H.: {\it Optimization and Nonsmooth Analysis,} Wiley, New York, (1983)
	
	
	
	
	
	
	
	
	
	
	
	
	
	
	
	
	
	
	
	
	
	
	
	
	
	
	
	
	
	
	
	
	
	
	
	
	
	
	
	
	
	
	
	
	
	
	
	
	\bibitem{Dai} Dai, G.: \textit{Nonsmooth version of Fountain theorem and its application to a Dirichlet-type differential inclusion problem}, Nonlinear Anal., 72, 1454-1461, (2010)
	
	\bibitem{Degiovanni} Degiovanni, M., Magrone, P.: \textit{Linking solutions for quasilinear equations at critical growth involving the
	1-Laplace operator}, Calc. Var. 36, 591–609 (2009)

	\bibitem{Demengel}  Demengel, F.: \textit{On some nonlinear partial differential equations involving the 1-Laplacian and
	critical Sobolev exponent}, ESAIM Control Optim. Calc. Var. 4, 667-686 (1999)

	\bibitem{Ekeland} Ekeland, I.: \textit{Nonconvex Mimnimization Problems}. Bull. Amer. Math. Soc., 443-474, (1979)
	
	\bibitem{Evans} Evans, L. C.: \textit{Partial differential equations}. American Mathematical Society, United States of American, (1998)
	
	\bibitem{Pimenta-Fig1} Figueiredo, G.M., Pimenta, M.T.O.: \textit{Strauss’ and Lions’ type results in $BV(\mathbb{R}^N)$ with an application to
		1-Laplacian problem}, Milan J. Math., 86, 15-30, (2018)
	
	\bibitem{Pimenta-Fig} Figueiredo, G.M., Pimenta, M.T.O.: \textit{Existence of bounded variation solutions for a 1-Laplacian problem
	with vanishing potentials}, J. Math. Anal. Appl., 459, 861-878, (2018)
	
	\bibitem{Gu} Gu, L., Zhou, H.: \textit{An improved fountain theorem and its application}, Adv. Nonlinear Stud. 17, 727 -738, (2016)
	
	\bibitem{Heinz} Heinz, H.: \textit{Free Ljusternik-Schnirelman Theory and the Bifurcation Diagrams of Certain Singular Nonlinear Problems}, J. Differential Equations  66, 263-300, (1987)
	
	\bibitem{Ji-Szulkin} Ji, C., Szulkin, A.: \textit{A logarithmic Schrödinger equation with asymptotic conditions on the potential}. J. Math. Anal. Appl. 437, 241-254, (2016)
	
	\bibitem{Kawohl} Kawohl, B., Schuricht, F.: \textit{Dirichlet problems for the 1-Laplace operator, including the eigenvalue problem},
	Commun. Contemp. Math. 9(4), 525–543 (2007)
	
	
	\bibitem{Kobayashi} Kobayashi, J., Ôtani, M.: \textit{The principle of symmetric criticality for non-differentiable mappings}, J. Funct. Anal., 214, 428-449, (2004)
	
	\bibitem{Lindestrauss} Lindenstrauss, J., Tzafriri, L.: \textit{Classical Banach Spaces I}, Springer, Berlin, (1977)
	
	\bibitem{Liu} Liu, D.: \textit{On a p-Kirchhoff equation via Fountain Theorem and Dual Fountain Theorem}, Nonlinear Anal. 72, 302-308, (2010)
	
	\bibitem{Mancini} Mancini, G.; Musina, R.: \textit{Hole and Obstacles}, Ann. Inst. H. Poincaré Anal. Non Linéaire, 4, 323-345, (1988)
	
	\bibitem{Molino}  Molino Salas, A.,  Segura de León, S.:\textit{ Elliptic equations involving the 1-Laplacian and a
	subcritical source term}, Nonlinear Anal., 168, 50–66, (2018)
	
	\bibitem{Motreanu} Motreanu, D., Panagiotopoulos, P. D.:\textit{Minimax Theorems and Qualitative Properties of The Solutions of Hemivariational Inequalities}, Springer Science $+$ Business Media Dordrecht, (1999)
	
	\bibitem{Nachbin} Nachbin, L.: \textit{The Haar Integral}, D. Van Nostrand Company, Canada, (1965)

	\bibitem{Ortiz}  Ortiz Chata, J. C.,  Pimenta, M. T. O.: \textit{A Berestycki-Lions’ type result to
	a quasilinear elliptic problem involving the 1-laplacian operator}, J. Math. Anal. Appl.
	doi.org/10.1016/j.jmaa.2021.125074.
	
	\bibitem{peral} Peral Alonso, I.: {\it Multiplicity of solutions for the $p$-laplacian}, Second School of Nonlinear Functional Analysis and Applications to Differential Equations, Trieste, 1997.

		
	\bibitem{Rabinowitz} Rabinowitz, P.H.: \textit{Minimax Methods in Critical Point Theory with Application to Diferential Equations}. American Mathematical Society, CBMS, 65, United States of American, (1986).
	
	\bibitem{Sun} Sun, J.: \textit{Infinitely many solutions for a class of sublinear Schrödinger-Maxwell equations}, J. Math. Anal. Appl., 390, 514-522, (2012)
	
	\bibitem{Squassina-Szulkin} Squassina, M., Szulkin, A.:\textit{Multiple solution to logarithmic Schrödinger eqautions with periodic potential}, Calc. Var. 54, 585-597 (2015)
	
	
	\bibitem{Szulkin} Szulkin, A.: \textit{Minimax principles for lower semicontinuous functions and applications to nonlinear boundary value problems}, Ann. Inst. H. Poincaré Anal. Non Linéaire 3, 77-109 (1986)
	
	
	\bibitem{Willem} Willem,  M.: {\it Minimax Theorems}, Birkh\"auser, Boston, (1996)
	
	\bibitem{Zou} Zou, W.: \textit{Variant fountain theorems and their applications}, Manuscripta Math. 104, 343-358 (2001)
	
	
	
	
\end{thebibliography}
\end{document}